\title{Reaching the minimum ideal in a finite semigroup}
\author{Nasim Karimi \footnote{Universidade Federal de Alagoas, Campus A. C. Simões - Av. Lourival Melo Mota, s/n, Cidade Universit\'aria, Macei\'o, Alagoas, 57072-900, Brasil
	Tel:  +55 (82) 3214-1418, 
Fax: 	+55 (82) 3214-1418 }\\
nakareme@gmail.com
}

\date{\today}
  
\documentclass[11pt]{article}

\usepackage{hyperref}
\usepackage{amsmath,amsthm,amsfonts,amssymb,amscd}
\usepackage[svgnames]{xcolor}
\usepackage{mathrsfs}
\usepackage{eepic}
\usepackage{amsmath}
\usepackage{imakeidx}
\newtheorem{theorem}{\bf Theorem}[section]

\newtheorem{conj}[theorem]{\bf Conjecture}
\newtheorem{prop}[theorem]{\bf Proposition}
\newtheorem{question}[theorem]{\bf Question}

\newtheorem{lem}[theorem]{\bf Lemma}
\newtheorem{defi}[theorem]{\bf Definition}
\newtheorem{rem}[theorem]{\bf Remark}
\newtheorem{cor}[theorem]{\bf Corollary}
\newtheorem{ex}[theorem]{\bf Example}
\def\rank{\ensuremath{\mathrm{rank}}}
\def\Im{\ensuremath{\mathrm{Im}}}

\def\diam{\ensuremath{\mathrm{diam}}}

\def\ord{\ensuremath{\mathrm{ord}}}
\def\dom{\ensuremath{\mathrm{Dom}}}

\def\deout{\ensuremath{\mathrm{d^{out}}}}
\def\dein{\ensuremath{\mathrm{d^{in}}}}
\newtheorem{notation}[theorem]{\bf Notation}

\begin{document}
\maketitle
\begin{abstract}
We introduce the depth parameters of a finite semigroup, which measure how hard it is to produce an element in the minimum ideal when we consider generating sets satisfying some minimality conditions. We estimate such parameters for some families of finite semigroups, and we obtain an upper bound for wreath products and direct products of two finite (transformation) monoids.

 \end{abstract}

\section{Introduction}
Consider a finite semigroup $S$ with a generating set $A$. Every element in $S$ can be represented as a product of generators in $A$. By the length  of an element $s$ in $S$, with respect to $A$, we mean the minimum length of a sequence which represents $s$ in terms of generators in $A$. In finite semigroup (group) theory, several parameters may be
defined involving the length
of elements in terms of a
generating set. In this work we are interested in the minimum length of elements in the minimum ideal (kernel) of a finite semigroup. 
We denote this parameter  by $N(S,A)$, where $A$ is a generating set of the finite semigroup $S$, and we call it $A$-depth of $S$. We define the following  parameters, called depth parameters, which 
 depend only on the semigroup $S$,
$$N(S)=\min \{ N(S,A): S=\langle A \rangle, ~\rank(S)=|A|  \},$$

$$M(S)=\max \{ N(S,A): S=\langle A \rangle, ~\rank(S)=|A| \}$$

and
$$N'(S)=\min \{ N(S,A): \mbox{$A$ is a minimal generating set}\},$$
$$M'(S)=\max \{ N(S,A):  \mbox{$A$ is a minimal generating set} \}.$$
Note that the minimum over all generating sets is zero in case of a group and is one otherwise, so it is of no interest. 

Part of our motivation to estimate such kind of parameters comes from  a famous conjecture in automata theory attributed to \v Cern\'y, a Slovak mathematician. In 1964, \v Cern\'y conjectured that any $n$-state synchronizing automaton has a reset word of length at most $(n-1)^2$ \cite{Cerny:1964}.  In fact, the transition semigroup of any finite automaton is a finite transformation semigroup. A reset word in a synchronizing automaton is a constant transformation, which belongs to the minimum ideal of the transition semigroup. Hence the length of a reset word in a synchronizing automaton is equal to the length of an element in the minimum ideal of the transition semigroup, with respect to a generating set. Also, there is a generalization of \v Cern\' y's conjecture, known as the
\v Cern\' y-Pin conjecture, which gives the upper bound $(n - r)^2$ for the length of a word
of rank $r$ in an automaton with $n$ states in which the minimum rank of
words is $r$. This version of the conjecture is a
reformulation of the
stronger conjecture in
\cite{Pin:1978}, which was disproved
in \cite{Kari:2001}. Here the automaton is not necessarily synchronizing  but the words of minimum rank $r$ represent elements in the minimum ideal of the transition semigroup. 

 We are also interested in investigating how the parameter $N(S,A)$ behaves with respect to the wreath product. In fact, the prime decomposition theorem states that any finite semigroup $S$ is a divisor of an iterated wreath product of its simple group divisors and the three-element monoid $U_2$ consisting of two right zeros and one identity element \cite{Rhodes&Steinberg:2009}. So, it should be interesting to be able to say something about $N(S,A)$ provided that $S$ is a wreath product of two finite transformation semigroups.

 In Section \ref{se:depth} we estimate the depth parameters for some families of finite semigroups. More precisely, we establish that the depth parameters are 
 equal, considerably small and easily calculable for any finite $0$-simple semigroup. We show that semilattices have a unique minimal generating set. So, the depth parameters 
 for semilattices are equal and again easily calculable. The third family of semigroups which we have considered is that of completely regular semigroups. For them the problem is reduced to the semilattice case.
 Afterward, we deal with transformation semigroups. We present in Theorem \ref{lower-bound-N'} a lower bound for $N'(S)$, where $S$ is any finite transformation semigroup, and we show that 
 it is sharp for several families of such semigroups. Applying this lower bound helps us to estimate the depth parameters for the transformation semigroups $PT_n$, $T_n$ and $I_n$; their ideals $K'(n,r)$, $K(n,r)$ and $L(n,r)$; and 
 the semigroups of order preserving transformations $PO_n$, $O_n$ and $POI_n$. The main theorem in that section is Theorem \ref{lower-bound-N'} which is proved by two 
 easy lemmas based on simple facts about construction of finite semigroups. Moreover, we use several results concerning the generating sets of minimum size of finite transformation semigroups (see for example 
 \cite{Fernandes:2001,Garba:1990,Gomes&Howie:1987,Gomes&Howie:1992,Howie&McFadden:1990,Howie:1966}).

  In Section \ref{se:product} we are interested in the behavior of the parameter $N(S)$ with respect to the wreath product and the direct product. For instance,
 we establish some lemmas to present a generating set of minimum size for the direct product (wreath product) of two finite monoids (transformation monoids). We compute the rank of the products (direct product or wreath product)
 in terms of their components.\footnote{The interested reader may find related results in \cite{Ruskuc&Robertson&Thomson:2003}.} Applying those results we give an upper bound for $N(S)$ where
 $S$ is a wreath product or direct product of two finite transformation monoids. 

\section{Preliminaries}\label{se:pre}
In this section we present the notation and definitions which we use in the sequel. For standard terms in semigroup theory see \cite{Pin:1986}.
\subsection{Depth parameters }
In this work we are only interested in non-empty finite semigroups.  
We note  that every finite semigroup has a minimum ideal which we call the \emph{kernel} of $S$ and  denote by $\ker(S)$. A non-empty subset $A \subseteq S$ is a \emph{generating set}, if every element in $S$ can be represented as a product of elements (generators) in $A$. We use the notation $S=\langle A \rangle$ when $A$ is a generating set of $S$. A generating set $A$ is called \emph{minimal} if no proper subset of $A$ is a generating set of $S$. 
By the \emph{rank} of a semigroup $S$, denoted by $\rank(S)$, we mean the cardinality of any of the smallest generating sets of $S$.\footnote{ When $S$ is a non trivial finite group, our notion of (semigroup) rank coincides with the notion of rank used in group theory (which allows the use of inverses) since the inverse of an element $a$ equals necessarily some power of $a$.  The rank of the trivial group is one by convention.}

 We suppose that the reader is familiar with the Green relations in the classical theory of finite semigroups. For a convenient reference see \cite{Pin:1986}. 
\begin{rem}
We use the fact that $\mathcal{J}=\mathcal{D}$ for a finite semigroup (the equality may fail for an infinite semigroup) several times in our proofs without mentioning it explicitly.
\end{rem}

\begin{defi}\label{length}
 Let $S$ be a finite semigroup with a generating set $A$. For every non identity element $s \in S $, the \emph{length} of $s$ with respect to $A$, denoted by $l_A(s)$, is defined to be 
 $$l_A(s):=\min\{k : s=a_1 a_2 \cdots a_k, ~ \mbox{for some} ~ a_1, a_2, \ldots, a_k \in A \},$$
 and the length of the identity (if there is any) is zero by convention. 
Furthermore, for any non empty subset $T$ of $S$, 
the maximum (minimum) length of $T$ with respect to $A$, denoted by $Ml_A(T)$ ($ml_A(T)$), is the maximum (minimum) length of elements, with respect to $A$, in $T$.
\end{defi}
 
\begin{defi}
Let $S$ be a finite semigroup with a generating set $A$. By the \emph{$A$-depth} of $S$ we mean the number
\begin{align*}
N(S,A):=ml_A(\ker(S)).
\end{align*}
\end{defi}
We may consider  the following parameters, defined in terms of the notion of $A$-depth, but which depend only on $S$: 
\begin{defi}
Let $S$ be a finite semigroup. Define
\begin{align*}
N(S)&:=\min \{ N(S,A) : S=\langle A \rangle , |A|=\rank(S) \},\\
N'(S)&:= 
\min\{ N(S,A) :~  A \mbox{ is a minimal generating set} \}\\ 
M(S)&:=\max \{ N(S,A) : S=\langle A \rangle , |A|=\rank(S) \},\\
M'(S)&:=\max \{ N(S,A) :~A \mbox{ is a minimal generating set} \}.
\end{align*}
\end{defi}
These are henceforth called the
\emph{depth parameters} of $S$.

\begin{ex}
If $G$ is a group then $N(G,A)=0$, for every generating set $A$ of $G$. Hence, all the depth parameters of $G$ are equal to zero.
\end{ex}

\begin{rem}
Note that the minimum $A$-depth over all generating sets of a finite semigroup which is not a group is one. 
\end{rem}
\begin{rem}
If $A \subseteq B$, then $N(S,B) \leq N(S,A)$. Hence, we have

$$M'(S)= 
\max \{ N(S,A) :~  A \mbox{ is a generating set} \}.$$
\end{rem}
\begin{rem}
It is easy to see that
$$N'(S) \leq N(S) \leq M(S) \leq M'(S).$$
\end{rem}
\begin{notation}
 Let $i \geq 1,~n \geq 1$  and  $C_{i,n} :=\langle a : a^i= a^{i+n}\rangle $ be the \emph{monogenic semigroup} with index $i$ and period $n $. 
\end{notation}
\begin{ex}
 For $i > 1$ we have $N(C_{i,n}, A)=i$ for every minimal generating set $A$ of $C_{i,n}$. Hence all the depth parameters are equal for all finite monogenic semigroups with index $i>1$.
\end{ex}
 \subsection{Semilattices}
A \emph{semilattice} is a semigroup $(S,.)$ such that, for any $x,y \in S$, $x^2=x$ and $xy=yx$. Given a semilattice $(S,\cdot)$ and $x,y \in S$, we define $x \leq y$ if $x = xy$. It is easy to see that $(S,\leq)$ is a partially ordered set that has a meet (a greatest lower bound) for any nonempty finite subset, indeed $x \wedge y = xy$  \cite{Almeida:1994}.

%

\begin{ex}
  Let $X$ be a set. The set $P(X)$ (set of subsets of $X$) with the binary operation of union is a semigroup. Since this semigroup is a free object in the variety of semilattices we call it the free semilattice generated by $X$.
\end{ex}

\begin{defi}
  Let $S$ be a semilattice. An element $s\in S$ is \emph{irreducible} 
  if $s=a b~ (a,b\in S)$ implies $a=s \mbox{ or } b=s$. Denote
  by $I(S)$ the set of all irreducible elements of $S$.
\end{defi}

  Let $(S,\leq)$ be a partially ordered set. As usual, let $<$ be the relation on $S$ such that $u < v$ if and only if $u \leq v$ and $u \not = v$. Let $u,v$ be elements of $S$. Then $v$ covers $u$, written $u \prec v$, if $u < v$ and there is no element $w$ such that $u < w < v $. By the \emph{diagram} of $(S,\leq)$ we mean the directed graph with vertex set $S$ such that there is an edge $u \rightarrow v$ between the pair $u,v\in S$ if $u\prec v$.

\begin{notation}
Given a vertex $v$ of a directed graph, the \emph{in-degree} of $v$ 
denoted by $\dein(v)$, is the number of $w$ such that
$(w,v)$ is an edge; the \emph{out-degree} of $v$, denoted by $\deout(v)$, is the number of $w$ such that $(v,w)$ is an edge.
\end{notation}
\begin{rem}
  Consider a finite semilattice $S$. By definition, the set $S$ has
 an infimum, which is the zero of $S$. Notice that in the diagram of
 $S$, the vertex corresponding to zero is the unique vertex which has in-degree zero. 
 \end{rem}

\begin{rem}
  Consider a finite semilattice $S$ with the property that the subset
  $\{x\in S : x\leq s\}$ is a chain for all $s\in S$. Then the diagram
  of $S$ is a rooted tree in which the root represents the
  zero of $S$.
\end{rem}
\subsection{Transformation semigroups}
\begin{notation}
Let $\mathbb{N}$ be the set of all natural numbers. For $n \in \mathbb{N}$ denote by $X_n$ the chain with $n$ elements, say $X_n=\{1,2,\ldots,n\}$ with the usual ordering.
\end{notation}
As usual, we denote by ${PT}_n$ the semigroup of all partial functions of $X_n$ (under composition) and we call the elements of $PT_n$  transformations.
We introduce two formally different (yet equivalent) definitions of a transformation semigroup:
   
\begin{defi}\label{transformation:defi1}
By transformation semigroup, with degree $n$, we mean a subsemigroup of the partial transformation semigroup ${PT}_n$.  
\end{defi}
Let $S$ be a finite semigroup and $X$ be a finite set. The semigroup $S$ faithfully acting on the right of the set $X$ means that there is a map $X \times S \rightarrow X$, written $(x,s) \mapsto xs$, satisfying:
\begin{itemize}
\item $x (s_1 s_2)= (xs_1) s_2$;
\item If for every $x \in X$ $xs_1=xs_2$, then $s_1=s_2$.
\end{itemize}
\begin{defi}\label{transformation:defi2}
 By a transformation semigroup $(X,S)$ we mean a semigroup $S$ faithfully acting on the right of a set $X$. 
\end{defi}

We define the families of transformation semigroups whose $A$-depth is estimated in \ref{su:depth}. 
Define the full transformation semigroup $T_n$ and the symmetric inverse monoid $I_n$ as follows:
\begin{align*}
T_n&:=\{ \alpha \in PT_n : \dom(\alpha)=X_n \},\\ 
I_n&:= \{ \alpha \in PT_n  : \alpha~ \mbox{is an injective transformation} \}.
\end{align*}
We further define certain transformation semigroups which are subsemigroups of $PT_n,T_n$ or $I_n$.
For instance, for $1 \leq r < n$ the following semigroups are ideals of $PT_n, T_n$ and $I_n$, respectively:
 
\begin{align*}
K'(n,r)&:= \{ \alpha \in PT_n : \rank (\alpha) \leq r \},\\ 
K(n,r)&:= \{ \alpha \in T_n : \rank (\alpha) \leq r \},\\
L(n,r)&:= \{ \alpha \in I_n : \rank(\alpha) \leq r \}.
\end{align*}
Also, we can define more transformation semigroups when we impose that the (partial) transformations to be order preserving.  
We say that a transformation $s$ in ${PT}_n$ is \textit{order preserving} if, for all $x,y \in \dom(s)$, $x \leq y$ implies $xs \leq ys$. 
Clearly, the product of two order preserving transformations is an order preserving transformation.

Let
\begin{align*}
PO_n&:= \{ \alpha \in PT_n \setminus \{1\} : \alpha \mbox{ is order preserving} \},\\
O_n&:=\{ \alpha \in T_n \setminus \{1\} : \alpha \mbox{ is order preserving} \},\\ 
POI_n&:= \{ \alpha \in I_n \setminus \{1\} : \alpha \mbox{ is order preserving} \}.\\
\end{align*}

Note that $PO_n$, $O_n$ and $POI_n$ are aperiodic semigroups (i.e., have trivial $\mathcal{H}$-classes). Denote by $J_{n-1}(PO_n), ~J_{n-1}(O_n)$ and $J_{n-1}(POI_n)$  the maximum $\mathcal{J}$-class in $PO_n,~O_n$ and $POI_n$, respectively. The $\mathcal{J}$-classes $J_{n-1}(PO_n)$, $J_{n-1}(O_n)$ and $J_{n-1}(POI_n)$ have $n$ $\mathcal{L}$-classes which consist of (partial) transformations of rank $n-1$ with  the same image. The $\mathcal{J}$-class $J_{n-1}(PO_n)$ has two kinds of $\mathcal{R}$-classes,  $n$ $\mathcal{R}$-classes consisting of proper partial transformations of rank $n-1$ and $n-1$ $\mathcal{R}$-classes consisting of total transformations of rank $n-1$; the $\mathcal{J}$-class $J_{n-1}(O_n)$  has $n-1$ $\mathcal{R}$-classes consisting of transformations of rank $n-1$; and the $\mathcal{J}$-class $J_{n-1}(POI_n)$ has $n$ $\mathcal{R}$-classes consisting of proper partial transformations of rank $n-1$.  

\subsection{Finite automata and $A$-depth of a semigroup}\label{se:finite automata}
We follow in this section the terminology of \cite{Rystsov:1992}.
 
A \emph{finite automaton} is a pair $A=(Q,\Sigma)$, where $Q$ is a finite state set and $\Sigma$ is a finite set of input symbols, each associated with a mapping on the state set $\sigma:Q \longrightarrow Q$ (note that we use the same notation for the symbols in $\Sigma$ and the associated mappings). A sequence of input symbols of the automaton will be called for brevity an \emph{input word}. To every input word $w=\sigma_1\sigma_2 \ldots \sigma_k$ is associated a mapping on the state set, which is a composition of the mappings corresponding to $\sigma_i$, $1 \leq i \leq k$. By the action of an input word we mean the action of the associated mapping. The action of the input word $w$ on the state $q$ is denoted $(q)w$ and the action of the input word $w$ on the subset of states $T$ is denoted $(T)w$.  Denote by $S_A$ the \emph{transition semigroup} of $A$ generated by the associated mappings of input symbols. In fact, $(Q,S_A)$ is the transformation semigroup generated by $\Sigma$.  

\begin{defi}
The \emph{rank} of a finite automaton is  the minimum rank of its input words (the \emph{rank} of a mapping is the cardinality of its image). An input word of minimum rank is called \emph{terminal}. 
\end{defi}

A finite automaton with rank one is called synchronizing and every terminal word in a synchronizing automaton is a \emph{reset word}. It is clear that the minimum ideal of the transition semigroup $S_A$ consists of the terminal words of the automaton $A$. Meanwhile, the parameter $N(S_A,\Sigma)$ is the minimum length of terminal words in the automaton $A=(Q,\Sigma)$. In fact, to compute the number $N(S,A)$, where $S$ is a finite transformation semigroup with a generating set $A$, is equivalent to finding the minimum length of terminal words in a finite automaton with  transition semigroup $S$. The importance of knowing the length of the terminal words in a finite automaton is motivated by the two following conjectures attributed to \v Cern\'y and Pin, respectively.    

\begin{conj}\cite{Cerny:1964} 
Every $n$-state synchronizing automaton has a reset word of length at most $(n-1)^2$.
\end{conj}
\begin{conj}\label{pin-conj} 
Every $n$-state automaton of rank $r$ has a terminal word of length at most $(n-r)^2$.
\end{conj}
 We mention that Pin generalized the \v Cern\'y conjecture as follows \cite{Pin:1978}. Suppose $A=(Q,\Sigma)$ is an automaton such that some word $w \in \Sigma^*$ acts on $Q$ as a transformation of rank $r$. Then he proposed that there should be a word of length at most $(n-r)^2$ acting as a rank $r$ transformation. This generalized conjecture was disproved by Kari \cite{Kari:2001}. However, the above conjecture is a reformulation of the Pin conjecture that is still open (and that was introduced by Rystsov as being the Pin conjecture \cite {Rystsov:1992}).
\section{Depth parameters of some families of finite semigroups }\label{se:depth}

In this section we estimate the depth parameters for some families of finite semigroups. We start with $0$-simple semigroups. We establish that the depth  parameters are equal, considerably small and easily computable for any finite $0$-simple semigroup. Then we show that semilattices  have a unique minimal generating set. So, the depth parameters are equal and again easily computable. The third family of semigroups which we have considered is that of completely regular semigroups. For them, the problem is reduced to the semilattice case. 

In all of the above examples, we did not represent semigroups as transformation semigroups. On the other hand, representing  the elements of a semigroup as transformations make us able to do some calculations. In the next part of this section we deal with transformation semigroups. We present in Theorem
 \ref{lower-bound-N'} a lower bound for $N'(S)$, where $S$ is any
    finite transformation semigroup, and we show that it is sharp for
    several families of such semigroups. Applying this lower bound helps us to estimate the depth parameters for some families of finite transformation semigroups.
\subsection{Examples}\label{su:example}
The following lemma is an easy observation which we are going to use frequently.
\begin{lem}\label{minimal-generating-set} 
Let $S$ be a finite semigroup and $I$ be an ideal  of $S$. If $I$ is contained in the subsemigroup generated by the set $S \setminus I$, then every minimal generating set of $S$ must be contained in $S \setminus I$.
\end{lem}
\begin{proof}
Let $A$ be a minimal generating set of $S$. Suppose that $ a \in I \cap A $. Because $I$ is contained in the subsemigroup generated by the set $S \setminus I$, $a$ can be written as a product of elements in $S\setminus I$. Moreover, because $I$ is an ideal  and $A$ is a generating set, every factor of this product can be written as a product of generators in $A \setminus I$. Therefore, $a$ can be written as a product of elements in $A \setminus I$, which contradicts the minimality of $A$. This shows that $A \cap I = \emptyset$. Hence we have $A \subseteq S \setminus I$.       
\end{proof}

A semigroup $S$ is called \emph{$0$-simple} \index{$0$-simple} if it possesses a zero, which is denoted by $0$, if $S^2 \not = 0$, and if, $\{0\}$ and $S$ are the only ideals of $S$ \cite{Pin:1986}. 
 The $0$-simple semigroups are examples of semigroups whose parameters $M,N,M',N'$ are equal, considerably small and easily computable.
\begin{lem}
  If $S$ is a finite  $0$-simple semigroup then $$ N(S)=M(S)=M'(S)=N'(S) \leq 2.$$
\end{lem}

\begin{proof}
 
If $S$ is a finite $0$-simple
  semigroup then it is isomorphic to a regular Rees matrix semigroup
  \cite{Pin:1986}.
  Let $S = M^0[G,I,L,P]$ be represented as a Rees matrix semigroup over a group $G$, where $P$ is a regular matrix with entries from $G \cup \{0\}$. 
If $P$ does not contain any entry equal to $0$, then every generating set must contain the zero element (since the other elements do not generate it). Therefore $N(S)=M(S)=M'(S)=N'(S)=1$.
Suppose that $P$ does contain at least one $0$ entry. In this case, no minimal generating set can contain the zero element of $S$, since then ${0}$ forms an ideal  of $S$ and the subsemigroup generated by $S \setminus \{0\}$ contains $0$ (see Lemma
\ref{minimal-generating-set}). Let $A$ be any generating set of $S$. We show that there are at
  least two not necessarily distinct elements of $A$ whose product is
  $0$. Let for some $k\geq 2$
  $$
  (i_1,g_{i_1},j_1)(i_2,g_{i_2},j_2)\cdots(i_k,g_{i_k},j_k)=0.
  $$
  Then there exists $1\leq l< k$ such that $p_{j_li_{l+1}}=0$. Hence
  $$
  (i_l,g_{i_l},j_l)(i_{l+1},g_{i_{l+1}},j_{l+1})=0.
  $$
  Therefore there are two not necessarily distinct elements of
  $A$ whose product is $0$, which shows that $N(S,A)=2$. It follows that $$N(S)=M(S)=M'(S)=N'(S)=2.\qedhere$$ 
\end{proof}

 Let $S$ be a finite semilattice. 
We show that $I(S)$, the set of all irreducible elements of $S$, is the unique minimal generating set of $S$. This leads to the  equality of all parameters $M,N,M',N'$ . Then we find a sharp upper bound for $I(S)$-depth of $S$. Finally, the special case where the diagram of $S$ is a rooted tree is considered.

\begin{lem} \label{generating set-semilattice}
  Let $S$ be a semilattice. The set $I(S)$ is the unique minimal generating set of $S$. 
\end{lem}

\begin{proof}
  Let $A$ be a generating set. First we show that $I(S)\subseteq A$.
  Let $s \in I(S)$. If $s \notin A$ then $s$ is a product of some
  elements in $A$ none of which is equal to $s$. This is in
  contradiction with irreducibility of $s$. Hence, we have $s \in A$.

  Now, we show that $I(S)$ is a generating set of $S$. Let $s\in
  S\setminus I(S)$. Then there exist $a,b\in S$ such that $s=a\wedge
  b$ while $s\neq a , s\neq b$. If both $a,b$ are irreducible then we
  are done, otherwise we repeat this process for $a,b$. This process
  must end after a finite number of steps because $S$ is finite and the elements which are produced at
  each step are strictly larger than the elements encountered in the previous step.
\end{proof}

The following corollary is an immediate consequence of Lemma \ref{generating set-semilattice}.

\begin{cor}\label{semilattice-cor}
Let $S$ be a finite semilattice. Then $$N(S)=N'(S)=M(S)=M'(S)=N(S,I(S)).$$
\end{cor}

\begin{prop}\label{prop:semilattice}
  The inequality $N(S,I(S))\leq |I(S)|$ holds for every finite semilattice
  $S$. The equality holds if and only if $S$ is the free semilattice  generated by $I(S)$.
\end{prop}

\begin{proof}
  First we show that the product of all elements in $I(S)$ is
  zero.
  Let $I(S)=\{a_1,a_2,\ldots,a_n\}$ and denote $a_1a_2\cdots a_n$ by $t$.
  If $s\in S$, then there exist
  $a_{i_1},a_{i_2},\ldots,a_{i_k}\in I(S)$ such that
  $s=a_{i_1}a_{i_2}\cdots a_{i_n}$. Now, we have $st=ts=t$ because $S$
  is commutative and idempotent. Therefore, we have $t=0$.
  
   For the second statement, first suppose that  $S$ is the free semilattice generated by $I(S)$. We show that $N(S)=|I(S)|$. Since $I(S)$ is a generating set of $S$, there exist $a_{i_1},a_{i_2}, \ldots, a_{i_k} \in I(S)=\{a_1,a_2,\ldots ,a_n\}$ such that $a_{i_1}a_{i_2}\cdots a_{i_k}=0$; because $S$ is commutative and idempotent we can suppose the $a_{i_j}$'s to be  distinct. Therefore, by the preceding paragraph, we have $a_{i_1}a_{i_2}\cdots a_{i_k}=a_1a_2\cdots a_n=0$. Now, because $S$ is a free semilattice we have $$\{a_{i_1},a_{i_2},\ldots , a_{i_k}\}=\{a_1,a_2 ,\ldots ,a_n\}$$ so that $k=n$. 
	
	Conversely, assuming that $N(S,I(S))= |I(S)|$, we show that $S$ is the free semilattice generated by $I(S)$. Suppose 
	\begin{equation}\label{above}
	a_{i_1}a_{i_2}\cdots a_{i_k}= a_{j_1} a_{j_2} \cdots a_{j_{\ell}}.
	\end{equation}
	Let $\{ a_{i_{k+1}}, a_{i_{k+2}}, \ldots a_{i_n}\}$ be the set $I(S) \setminus \{a_{i_1},\ldots, a_{i_k}\}.$ By equality \eqref{above}, we have 
\begin{equation*}
a_{i_1}a_{i_2}\cdots a_{i_k}a_{i_{k+1}} a_{i_{k+2}}\cdots a_{i_n}= a_{j_1}a_{j_2}\cdots a_{j_l}a_{i_{k+1}} a_{i_{k+2}}\cdots a_{i_n}.
\end{equation*}
	Since $N(S)=M'(S)= |I(S)|$ the subset $\{ a_{j_1},a_{j_2},\ldots a_{j_{\ell}},a_{i_{k+1}}, a_{i_{k+2}},\ldots a_{i_n}\}$ must be the whole set $I(S)$. 
	This shows that $$\{a_{i_1},a_{i_2},\ldots ,a_{i_k}\} \subseteq \{a_{j_1}, a_{j_2} \ldots ,a_{j_{\ell}} \}.$$ By
symmetry, the reverse
inclusion $\{a_{j_1},\ldots, a_{j_{\ell}}\} \subseteq\{a_{i_1},\ldots, a_{i_k}\}$ also holds.  It follows that $S$ is the semilattice freely generated by $I(S)$.
	\end{proof}

\begin{prop}
  If the diagram of a finite semilattice $S$ is a rooted tree then
  $N(S,I(S))\leq 2$.
\end{prop}
  
\begin{proof} 
  Denote the diagram of $S$ by $T$. It is clear that
  $I(S)=\{ v \in V(T) : 
  \deout(v)\leq 1\}.$
    Let $v_0$ be the root of the tree $T$. If $v_0$ belongs to $I(S)$ then
    $N(S,I(S)) \leq 1$. Suppose that $v_0 \not \in I(S)$. We show that there are
    two elements in $I(S)$ whose product is zero.
    Because $\deout(v_0)\geq 2$, there exist two distinct vertices
    $v_1,v_2$ such that $v_0 \rightarrow v_1$ and $v_0 \rightarrow v_2$.
Denote by $T_i$ the rooted subtree of $T$ with $v_i$ as its root. Note that $V(T_i)\cap I(S)\neq \varnothing$ because every subtree contains leaves and leaves are irreducible. If $u_i$ belongs to  $V(T_i)\cap I(S)$ then $u_1u_2=0$.
\end{proof}

Let $S$ be a completely regular semigroup. Green's relation
$\mathcal{D}$ is a congruence in $S$ and $S/ \mathcal{D}$ is a
semilattice of $\mathcal{D}$-classes which are simple semigroups
\cite{Higgins:1992}. Hence, by the results obtained for semilattices, we have the following lemma for completely regular semigroups.

 If a $\mathcal{D}$-class of a completely regular semigroup $S$ is an
irreducible element of the semilattice $S/ \mathcal{D}$, then we call it an \emph{irreducible
  $\mathcal{D}$-class} of $S$\index{irreducible!${D}$-class}. Denote by $\mathrm{IRD}(S)$\index[notation]{$\mathrm{IRD}(S)$} the set of
all irreducible $\mathcal{D}$-classes of $S$.

\begin{lem}\label{completely-regular} 
  Let $S$ be a completely regular semigroup. Then the following inequality holds
  $$M'(S) \leq N(S/\mathcal{D} )\leq |\mathrm{IRD}(S)|.$$
\end{lem}

\begin{proof}
  Let $A$ be a generating set of $S$. First we show that $ D \cap
  A\not = \varnothing $ for every $D\in \mathrm{IRD}(S)$. Let $D \in
  \mathrm{IRD}(S)$ and $d\in D $. There exist $a_1,a_2,\ldots,a_j \in
  A$ such that $d=a_1a_2\cdots a_j$. Therefore, we have $D_{a_1}D_{a_2}\cdots
  D_{a_j}\subseteq D_d=D$. Because $D$ is an irreducible
  $\mathcal{D}$-class of $S$ there exists $k\in \{1,2,\ldots,j\}$ such
  that $D_{a_k}=D$. Therefore, we have $a_k\in D\cap A \not = \varnothing$.

Now, we prove the first inequality. Let $t=N(S/ \mathcal{D})$. By Corollary \ref{semilattice-cor}, there are irreducible $\mathcal{D}$-classes $D_1,D_2,\ldots,D_t$ of $S$ such that 
$D_1 D_2\cdots D_t = \ker(S)$. Let $a_i \in A \cap D_i$ (we have shown that it exists). Then $a_1 a_2 \cdots a_t \in \ker(S)$, whence $N(S,A) \leq t$. Since $A$ is arbitrary, we 
get $M'(S) \leq N(S/ \mathcal{D})$. The second inequality follows from Proposition \ref{prop:semilattice}.
\end{proof}
\subsection{$A$-depth of transformation semigroups}\label{su:depth}

Our main goal in this section is estimating the depth parameters for some families of finite transformation semigroups. 
First we find a lower bound for $N'(S)$ where $S$ is any finite transformation semigroup. Let $S$ be a finite transformation semigroup and $A$ be a minimal generating set of $S$. Denote by $r(S,A)$\index[notation]{$r(S,A)$} the minimum of the ranks of elements in $A$; and denote by $t(S)$\index[notation]{$t(S)$} the rank of elements in the minimum ideal of $S$. The following corollary of Lemma \ref{minimal-generating-set} shows that $r(S,A)$ is independent of the choice of the  minimal generating set $A$.
\begin{cor}\label{r-minimal}
Let $S \leq PT_n$ be a finite transformation semigroup. Let $A$ and $B$ be two minimal generating sets of $S$. We have $r(S,A) = r(S,B).$
\end{cor}
\begin{proof}
It is enough to show that
 $\min \{\rank(f) : f \in A \} \leq  \min \{\rank(f): f \in B \}.$
 Suppose that $\min \{\rank(f) : f \in A\}=r$. If $\{f \in S : \rank(f) < r\}=\emptyset$ then we are done. Let $\{f \in S : \rank(f) < r\} \not =\emptyset$. Consider the subsemigroup 
$I=\{f \in S : \rank(f) < r\}.$ It is easy to see that $I$ is an ideal of $S$. Since $A \subseteq S \setminus I$, by Lemma \ref{minimal-generating-set} we have $B \subseteq S \setminus I$. Hence, we have  $\min \{\rank(f): f \in B \} \geq r$.   
\end{proof}
From now on, we use $r(S)$\index[notation]{$r(S)$} instead, since it depends only on $S$.

 \begin{lem}\label{rank-lower-bound1}
Let $X=\{ f \in {PT}_n : \rank(f) \geq r \}$. For $f_1,f_2,\ldots,f_k \in X$ the inequality  
\begin{equation}\label{eq-rank-lower-bound}
\rank(f_1 f_2 \cdots f_k) \geq n-k(n-r),
\end{equation}
holds.
 \end{lem}

\begin{proof}
We use induction on $k$. For $k=1$, the lower bound given by \eqref{eq-rank-lower-bound} is obvious. Now, let $f_1,f_2, \ldots ,f_{k+1}$ be $k+1$ not necessarily distinct elements of $X$. Denote the composite transformation $f_1f_2 \cdots f_k$ by $f$. By the induction hypothesis, we know that $\rank(f) \geq n-k(n-r)$. Then, it is enough to show for $f_{k+1} \in X$ that
$$\rank(ff_{k+1}) \geq n-(k+1)(n-r).$$
Let $\rank(f)=t$ and $\Im(f)=\{a_1,a_2,\ldots,a_t\}$. Suppose that $$\rank(ff_{k+1}) < n-(k+1)(n-r).$$ Because $\rank(f_{k+1}) \geq r$, it follows that 
$$|(X_n\setminus \{a_1,a_2,\ldots,a_t\})f_{k+1}| > r- (n-(k+1)(n-r)).$$ On the other hand, the inequality 
 $|(X_n\setminus \{a_1,a_2,\ldots,a_t\})f_{k+1}| \leq n-t$ holds. Hence $$r- (n-(k+1)(n-r)) < n-t$$ which gives $t < n-k(n-r)$. Since $t=\rank(f) \geq n-k(n-r)$, this contradiction implies that $$\rank(ff_{k+1}) \geq n-(k+1)(n-r),$$ which completes the proof. 
\end{proof}

The next theorem gives a lower bound for $N'(S)$ where $S$ is a finite transformation semigroup. 
\begin{notation}
For any number $k$ denote by $\left\lceil  k \right\rceil$ \index[notation]{$\left\lceil  k \right\rceil$} the least integer greater than or equal to $k$.  
\end{notation}

\begin{theorem} \label{lower-bound-N'}
If $S \leq {PT}_n$ and $S$ is not a group with $r(S) \leq n-1$, then $$N'(S) \geq \left\lceil \frac{n-t(S)}{n-r(S)}\right\rceil .$$
 
\end{theorem}
\begin{proof}
Let $A$ be a minimal generating set of $S$. Note  $A \subseteq \{ f \in S : \rank (f) \geq r(S) \}.$ Let $f_1,f_2,\ldots,f_k \in A$ such that $f_1f_2\ldots f_k \in\ker(S)$. Since $\rank(f_1f_2 \ldots f_k)=t(S),$ then by Lemma \ref{rank-lower-bound1} we have $k \geq \left\lceil \frac{n-t(S)}{n-r(S)}\right\rceil$. Hence, we have $N(S,A) \geq\left\lceil \frac{n-t(S)}{n-r(S)}\right\rceil $, which is the desired conclusion.
\end{proof}

Theorem \ref{lower-bound-N'} presents a lower bound for $N'$ for finite transformation semigroups which are not groups. For estimating the other parameters $N,M,M'$ we should know more about generating sets. Nevertheless, the following very simple lemma provides the main idea to estimate those parameters for some families of finite transformation semigroups.

\begin{lem}\label{maximum-j-class}
Let $S$ be a finite semigroup such that $S\setminus \{1\}$ \footnote{Note that $S\setminus \{1\}=S$ if $S$ is not a monoid.} is its subsemigroup and has a unique maximal $\mathcal{J}$-class $J$. Let $A$ be a  generating set of $S$. Then each $\mathcal{L}$-class and each $\mathcal{R}$-class of $J$ has at least one element in $A$.   
\end{lem}
\begin{proof}
Let $x\in J$. Since $S$ is  finite  we have $\mathcal{J}=\mathcal{D}$, then for $x$ to be a product of elements of $A$ it is necessary that at least one element of $A$ be $\mathcal{L}$-equivalent to $x$ and at least one element of $A$ be $\mathcal{R}$-equivalent to $x$. Thus $A$ must cover the $\mathcal{L}$-classes and also the $\mathcal{R}$-classes of $J$.  
\end{proof}
 
Now we are ready to apply the results in this section to the transformation semigroups  
$PT_n,T_n,I_n,$ their ideals $K'(n,r),K(n,r),L(n,r)$ and the semigroups of order preserving transformations $PO_n,O_n,POI_n$. If $S$ is one of the semigroups $PO_n, O_n$ or $POI_n$, then $S \setminus \{1\}$ is a subsemigroup of $S$ with a unique maximum $\mathcal{J}$-class \cite{Gomes&Howie:1992,Fernandes:2001}. Moreover, if $S$ is one of $K'(n,r),~K(n,r)$ or $L(n,r)$, then $S \setminus \{1\}=S$ has a unique maximum $\mathcal{J}$-class \cite{Howie&McFadden:1990,Garba:1990}. 
Hence, except for $T_n, PT_n$ and $I_n$ the above semigroups satisfy the hypothesis of Lemma \ref{maximum-j-class}. Thus, our strategy for estimating the depth parameters is different for these semigroups. First, we need to identify the generating sets of minimum size for $T_n, PT_n$, $I_n$. It is well known that, for $n \geq 3$,
 $$\rank(T_n)=3, \rank(I_n)=3, \footnote{Usually by a generating set of an inverse semigroup one means a subset $A \subseteq S$ such that every element in $S$ is a product of elements in $A$ and their inverses. But we do not include inverses here.} \rank(PT_n)=4.$$ But, we need to know exactly what are the generating sets of minimum size. So, we establish the following lemmas for completeness.
 
  \begin{notation}
We use the notation $(i,j)$ for denoting a transposition.
 \end{notation} 
 
\begin{lem}\label{generating-set-Tn}
   Let $ A=\{a,b,c\}\subseteq T_n~(n\geq 3)$ such that $\{a,b\}$ generates
  $S_n$ and $c$ is a function of rank $n-1$. Then, $ A $ is a
  generating set of $ T_n$ with minimum size. Furthermore, all
  generating sets of $ T_n$ with minimum size are of this form.
\end{lem}

\begin{proof}
  Since the symmetric group $ S_n $ cannot be generated by less than
  two elements for $ n\geq 3 $, we need at least three elements to
  generate $ T_n $. Then it suffices to show that such a set $A$
  generates $ T_n $. We know that every element of $T_n \setminus S_n$ is a product of idempotents of rank $n-1$ \cite{Howie:1966}. Therefore, we show that $A$ generates all idempotents of rank $n-1$ (because $ \{a,b,c\}$ already
 generates all permutations). Since $c$ is a function of rank $ n-1 $, there exist exactly two distinct numbers $ 1\leq i<j \leq n $
  such that $ic=jc=l$, and there exists a unique number $1\leq
  k\leq n$ such that $k\not \in \Im(c)$. Suppose that $\alpha$ is an idempotent of rank $n-1$, which implies that $\alpha$ has the form

$$
   \alpha=\left(
   \begin{array}{lllll}
    a_1 & a_2 &  a_3 & \ldots  & a_n\\
    a_1 & a_1 &  a_3 & \ldots  & a_n\\ 
    \end{array} \right),
$$
where
$\{1,2,\ldots,n\}=\{a_1,a_2,\ldots,a_n\}.$
Let  
$ \rho=\left(
    \begin {array} {ccccccccc}
      a_1 & a_2 & \ldots & a_i & \ldots & a_n\\
      1 & 2 & \ldots & i & \ldots & n\\
    \end{array}
  \right),
$
and define permutations $ \tau, \sigma$ as follows. If $i=2$ let
$\tau$ be the cycle $(i,j,1)$ and
\begin{equation}
t\sigma=\left \{
    \begin {array} {rrl}
     a_r & \mbox{if} & t=rc, r\not \in \{j,1,2\}\\
     a_1 & \mbox{if} & t=l\\
     a_2 & \mbox{if} & t=k\\
     a_j & \mbox{if} & t=1c.\\
    \end{array}  \right.
\end{equation}
If $i=1,~j=2$ let $\tau$ be the identity function and let 
\begin{equation}
t\sigma=\left \{
    \begin {array} {rrl}
     a_r & \mbox{if} & t=rc, r\not \in \{1,2\}\\
     a_1 & \mbox{if} & t=l\\
     a_2 & \mbox{if} & t=k.\\
    \end{array}  \right.
\end{equation}
In the remaining cases let $\tau=(i,1)(j,2)$ and let
\begin{equation}
t\sigma=\left \{
    \begin {array} {rrl}
     a_r & \mbox{if} & t=rc, r\not \in \{ i,j,1,2\}\\
     a_1 & \mbox{if} & t=l\\
     a_2 & \mbox{if} & t=k\\
     a_i & \mbox{if} & t=1c\\
     a_j & \mbox{if}& t=2c.\\
    \end{array}  \right.
\end{equation}
Now, it is easy to check that $\alpha= \rho \tau c \sigma.$
 
 The last statement of the lemma follows from the structure of
  $\mathcal{J}$-classes of $T_n$. More precisely, $J_{n-1}
  =\{f\in T_n: \rank(f)=n-1\} $ is a $ \mathcal{J}$-class of $T_n$ which is
  $\mathcal{J}$-above all the other $ \mathcal{J}$-classes except the
  maximum $ \mathcal{J}$-class. Therefore, every generating set of
  $T_n$ must have at least one element in the $ \mathcal{J}$-class
  $J_{n-1}$.
 \end {proof}

\begin{lem}\label{generating-set-PTn}
   Let $ A=\{a,b,c,d\}\subseteq {PT}_n ~(n\geq 3)$ such that $\{a,b,c\}$ generates
  $T_n$ and $d$ is a proper partial function of rank $n-1$. Then $ A $ is a
  generating set of $ {PT}_n$ with minimum size. Furthermore, all
  generating sets of $ {PT}_n$ with minimum size are of this form.
\end{lem}

\begin{proof}
  By Lemma \ref{generating-set-Tn}, the full transformation semigroup  $ {T}_n $ cannot be generated by less than
 three elements for $ n\geq 3 $. On the other hand, elements of $T_n$ cannot generate any proper partial function so we need at least four elements to
  generate $ {PT}_n $. Then it suffices to show that such a set $A$
  generates ${PT}_n $.  First, we prove this for the particular case in which 
  $$
     d=\left(
     \begin{array}{ccccc}
     1 & 2 & \ldots & n-1 & n\\
     - & 1 & \ldots & n-2 & n-1\\ 
     \end{array} \right).
     $$

Since $\{a,b,c\}$ generates $ T_n $, we must show that, by adding  $d$, we reach all proper partial functions. For $k \geq 1$, let
  $$
     f=\left(
     \begin{array}{cccccccc}
     a_1 & a_2 & \ldots & a_k & a_{k+1} & \ldots & a_n\\
     - & - & \ldots & - & b_{k+1} & \ldots & b_n  \\ 
     \end{array} \right)
     $$
     be a proper partial function which is undefined in exactly $k$ elements.
     Then, it is easy to check that $f=\sigma d^k g$ where $\sigma$ is the permutation
       $$
     \sigma=\left(
     \begin{array}{ccccccccc}
     a_1 & a_2 & \ldots & a_k & a_{k+1} & \ldots & a_n\\
     1 & 2 & \ldots & k & k+1 & \ldots & n\\ 
     \end{array} \right),
     $$
     and $g$ is the function
     $$
      g=\left(
     \begin{array}{ccccccccc}
     1 & 2 & \ldots & n-k & n-k+1 & \ldots & n\\
     b_{k+1}& b_{k+2} & \ldots & b_n & n-k+1 & \ldots & n\\ 
     \end{array} \right).
     $$
     For the general case, let 
     $$ d'=\left(
     \begin{array}{ccccccccc}
     a'_1 & a'_2 & \ldots & a'_n\\
     - & b'_2 & \ldots & b'_n\\ 
     \end{array} \right).
     $$
     where $$\{1,2,\ldots,n\}=\{a'_1,a'_2,\ldots,a'_n\}=\{b'_1,b'_2,\ldots,b'_n\}.$$
   We show that $\{a,b,c,d'\}$ generates $PT_n$. It is enough to show that $d$ is a product of elements in $\{a,b,c,d'\}$. Define the permutations $\rho,\delta$ as follows
      $$ \rho=\left(
     \begin{array}{ccccccccc}
     1 & 2 & \ldots & n\\
     a'_1 & a'_2 & \ldots & a'_n\\
     \end{array} \right),
     $$
     and
       $$ \delta=\left(
     \begin{array}{ccccccccc}
     b'_1 & b'_2 & \ldots & b'_n\\
     n & 1 & \ldots & n-1\\
     \end{array} \right).
     $$
     Now, it is easy to check that $d=\rho d' \delta$.
     
Finally, we show that all generating sets of ${PT}_n$ of minimum size are of the stated form. Let $A$ be any generating set of ${PT}_n$. Since $T_n \subseteq {PT}_n$ and $PT_n \setminus T_n$ is an ideal, then $A$ must contain a generating set of ${T}_n$. On the other hand,  elements of ${T}_n$ cannot generate any proper partial function. Therefore, $A$ must contain at least one proper partial function.  Since all the  proper partial functions of rank $n-1$ are in the $\mathcal{J}$-class which is $\mathcal{J}$-above all $\mathcal{J}$-classes but the maximum $\mathcal{J}$-class, then $A$ must contain at least one partial function of rank $n-1$.       
     \end{proof}
 
\begin{lem}\label{generating-set-In}
   Let $ A=\{a,b,c\}\subseteq I_n ~(n\geq 3)$ be such that $\{a,b\}$ generates
  $S_n$ and $c$ is an element of $J_{n-1}=\{ \alpha \in I_n : rank(\alpha)=n-1\}$. Then $ A $ is a generating set of $ I_n$ with minimum size. Furthermore, all
  generating sets of $ I_n$ with minimum size are of this form.
\end{lem}

\begin{proof}
We know that $\{a,b,c,c^{-1}\}$ is a generating set of $I_n$ \cite{Gomes&Howie:1987}. We only need to show that $c^{-1} \in \langle a,b,c\rangle.$ Let $\dom(c)=X_n \setminus \{i\}, \Im(c)=X_n \setminus \{j\}$.  For $i \not = j$, let $\alpha=(i,j)$ be a transposition and for $i=j$, let $\alpha$ be the identity function. We may complete $c^{-1}$ to an element $\theta$ of $S_n$ by defining $j \theta =i$. It is easy  to check that $\alpha c \theta \alpha \theta= c^{-1}$.
    
For the second statement, let $A$ be any generating set of $I_n$. Since $S_n$ is the maximum $\mathcal{J}$-class of $I_n$, $A$ must contain a generating set of $S_n$, which has at least $2$ elements for $n \geq 3$.
On the other hand, since $S_n$ is a group, the elements of $S_n$ are not enough to generate the whole semigroup $I_n$. So, we need at least one element in $I_n \setminus S_n$. Since $J_{n-1}$ is $\mathcal{J}$-above all
 $\mathcal{J}$-classes but the maximum $\mathcal{J}$-class, then $A$ must contain at least one element in $J_{n-1}$.  
\end{proof}

Part of the following corollary is immediate by Theorem \ref{lower-bound-N'}.     

\begin{cor}\label{N-transformation-cor1}
For $n \geq 3$,
\begin{align*}
 N'(T_n)&=N(T_n)=n-1,\\
  N'(PT_n)&=N(PT_n)=n,\\
   N'(I_n)&=N(I_n)=n.
\end{align*}
\end{cor}

\begin{proof}
 Since
 \begin{align*}
 &t(T_n)= 1,\\
  &t(I_n)=t(PT_n)=0,\\
  &r(T_n)=r(PT_n)=r(I_n)=n-1,
\end{align*}
then by Theorem \ref{lower-bound-N'},
\begin{align*}
N'(T_n) \geq n-1, ~N'(PT_n) \geq n , ~N'(I_n) \geq n.
\end{align*}
What is left is to show that
\begin{align*}
N(T_n) \leq n-1, ~N(PT_n) \leq n , ~N(I_n) \leq n.
\end{align*}
We do this by showing that each of the above semigroups has a generating set $A$ of minimum size for which $A$-depth is at most the proposed upper bound. 
By Lemma \ref{generating-set-PTn}, the rank of $PT_n$ is four and the set $A=\{\alpha,\beta,\theta,\gamma\}$ is a generating set of $T_n$ provided that $\{\alpha,\beta\}$ is a generating set of the symmetric group  $S_n$, $\theta$ is a transformation of rank $n-1$, and $\gamma$ is a proper partial transformation of rank $n-1$. If we choose $\gamma $ to be the partial transformation
$$ \gamma=\begin{pmatrix} 
1&2&3&\ldots&n\\
-&1&2&\ldots&n-1\\
\end{pmatrix},$$
then $\gamma^n$ is the empty map, which lies in the minimum ideal  of $T_n$. This shows that $N(PT_n,A) \leq n.$ 
With the above notation and by Lemma \ref{generating-set-In}, the set $A'=\{\alpha,\beta,\gamma\}$ is a generating set of $I_n$ of minimum size and the above argument gives $N(I_n, A') \leq n$. 
For $T_n $, again, with the above notation and by Lemma \ref{generating-set-Tn} the set $A=\{\alpha,\beta,\theta\}$ is a generating set of minimum size. If we choose $\theta $ to be the transformation
$$ \theta=\begin{pmatrix} 
1&2&3&\ldots&n\\
1&1&2&\ldots&n-1\\
\end{pmatrix},$$
then $\theta^{n-1}$ is the constant map, which lies in the minimum ideal  of $T_n$. Hence, we have $N(T_n,A) \leq n-1$.  
\end{proof}

Now we show that $N=N'$ for the remaining semigroups, and indeed $N'=N=M=M'$ (except for the semigroup $O_n$).
 \begin{prop}\label{N-transformation-cor3}
For $n \geq 3$, 
 \begin{align*}
 &N'(PO_n)=N(PO_n)=M(PO_n)=M'(PO_n)=n,\\
&N'(O_n)=N(O_n)=n-1,\\
   &N'(POI_n)=N(POI_n)=M(POI_n)=M'(POI_n)=n.
 \end{align*}
 \end{prop}
 \begin{proof}
 We start with the semigroup $PO_n$. We know that $PO_n $ is generated by the $\mathcal{J}$-class $J_{n-1}$ consisting of transformations or partial transformations of rank $n-1$ \cite{Gomes&Howie:1992}, and the empty transformation is the zero of $PO_n$. Hence, we have $r(PO_n)=n-1$ and $t(PO_n)=0$. So Theorem \ref{lower-bound-N'} implies that $N'(PO_n) \geq n$. It remains to show that $M'(PO_n) \leq n$. Let $A$ be a minimal generating set of $PO_n$. By Lemma \ref{maximum-j-class}, $A$ intersects each $\mathcal{R}$-class of $J_{n-1}$. Hence, we can find proper partial transformations $f_1,f_2,\ldots,f_n \in A$ such that $1 \not \in \dom(f_1)$ and for $1 \leq i \leq n-1,~( i+1)f_1f_2\ldots f_i \not \in \dom(f_{i+1})$. It is easy to see that $f_1f_2\ldots f_n$ is the empty function. This shows that $N(PO_n,A) \leq n$. Since $A$ is an arbitrary minimal generating set, then $M'(PO_n) \leq n.$
 
 The next semigroup in the statement of the proposition is the semigroup $O_n$. Since the maximum $\mathcal{J}$-class $J_{n-1}$ generates $O_n$ \cite{Gomes&Howie:1992}, $r(O_n)=n-1$. By Theorem \ref{lower-bound-N'}, $N'(O_n) \geq n-1$. We show that $N(O_n) \leq n-1$. It is enough to show that $N(O_n,A) \leq n-1$ for some generating set $A$ of minimum size. For $1 \leq  i \leq n-1$ let 
 \begin{equation*}
\alpha_i= \left (\begin{array}{cccccccccc} 
1&2& 3&\ldots &i&i+1& \ldots n\\
1&2& 3&\ldots &i+1&i+1&\ldots n
\end{array} \right )
 \end{equation*}
 and\
 \begin{equation*}
\beta= \left (\begin{array}{ccccccccccc} 
1&2& 3&\ldots &i&i+1& n-1&\ldots n\\
1&1& 2&\ldots &i-1&i&n-2&\ldots n-1
\end{array} \right ).
 \end{equation*}
 The set $\{ \alpha_1,\alpha_2,\ldots,\alpha_{n-1},\beta \} $ is a generating set of $O_n$ of minimum size as has been proved in \cite{Gomes&Howie:1992}. On the other hand,  
$\beta^{n-1}$ is a constant transformation. This shows that $N(O_n,A) \leq n-1$, and so $N(O_n) \leq n-1$.
 
We now apply this argument again, for $POI_n$.  Reasoning as in the previous cases, we obtain $N'(POI_n) \geq n$ \cite{Fernandes:2001}. We show that $N(POI_n,A) \leq n$ for every minimal generating set $A$. Again, $A$ intersects each $\mathcal{R}$-class of $J_{n-1}$. Hence, we can find proper partial transformations $f_1,f_2,\ldots,f_n \in A$ such that $1 \not \in \dom(f_1)$ and for $1 \leq i \leq n-1,~(i+1)f_1f_2\ldots f_i \not \in \dom(f_{i+1})$. It is easy to see that $f_1f_2\ldots f_n$ is the empty function. Hence, we have $N(POI_n,A) \leq n$, which completes the proof. 
 \end{proof}
 We use the following lemmas to prove Proposition \ref{N-transformation-cor2}. 
 \begin{lem}\label{L(n,r)}
 The transformation semigroup $L(n,r)$ is generated by its maximum $\mathcal{J}$-class. 
 \end{lem}
 \begin{proof}
  For $0 \leq k \leq r$ denote 
$$J_k:=\{ \alpha \in L(n,r): \rank(\alpha)=k\}.$$  It is easy to see that $J_k$ is a $\mathcal{J}$-class of $L(n,r)$. Now, we prove that the maximum $\mathcal{J}$-class 
$J_r$ generates $L(n,r)$. For $k < r$, consider an arbitrary $ \beta \in J_k$. Suppose that
\begin{equation*}
\beta=\left(
\begin{array}{ccccc}
a_1 & a_2 & \ldots &a_k\\
b_1& b_2 & \ldots &b_k
\end{array}
\right).
\end{equation*}
Choose $a_{k+1} \not \in \dom(\beta)$ and $b_{k+1} \not \in \Im(\beta)$ and let
 \begin{equation*}
 \beta'=\left(
\begin{array}{ccccc}
a_1 & a_2 & \ldots &a_k & a_{k+1}\\
b_1& b_2 & \ldots &b_k & b_{k+1} 
\end{array}
\right).
\end{equation*}
  Now, choose $ f \in J_r$ such that $f(b_i)=b_i ~\mbox{for}~ 1 \leq i \leq k$ and $b_{k+1} \not \in \dom(f)$. It is easy to see that $\beta=\beta'f$. Therefore, $J_k \subseteq J_{k+1}J_r$ for $0 \leq k \leq r-1$. It follows that $J_k \subseteq J_r^{r-k+1}.$ Hence, $L(n,r)$ is generated by $J_r$.
 \end{proof}

\begin{lem}\label{partial-permutation}
The $\mathcal{R}$-class in $K'(n,r)$ of
a partial permutation consists
only of partial permutations. Moreover,
 two partial
permutations which are $\mathcal{R}$-equivalent
in $K'(n,r)$ are also
$\mathcal{R}$-equivalent in $L(n,r)$.
\end{lem}
\begin{proof}
Let $f,g \in K'(n,r)$. Suppose that $f \mathcal{R} g$ and $f$ is a partial permutation. There exist $h,k \in K'(n,r)$ such that $f=gh$ and $g=fk$. First we  show that $g$ is a partial permutation. Since  $f$ is equal to $gh$, then $\dom(f) \subseteq \dom(g)$ and $\rank(f) \leq \rank(g)$. Since  $g$ is equal to $fk$, then  $\dom(g) \subseteq \dom(f)$ and $\rank(g) \leq \rank(f)$. Hence, we have $\dom(f)=\dom(g)$ and $\rank(f)=\rank(g)$. Since $f$ is a partial permutation, then $|\dom(f)|=\rank(f)$. It follows that $|\dom(g)|=\rank(g)$, hence $g$ is a partial permutation. Now, define the partial permutations $h',k'$ as follows. Let $\dom(h')=\Im(g)$ and $xh=xh'$ for every $x \in \Im(g)$. Let $\dom(k')=\Im(f)$ and $xk=xk'$ for every $x \in \Im(f)$. Hence, we have $f=gh'$ and $g=fk'$ and $h',k' \in L(n,r)$. This shows that $f,g$ are $\mathcal{R}$-equivalent in $L(n,r)$.
\end{proof}

\begin{prop}\label{N-transformation-cor2}
For every $n >1$ and $1 \leq r \leq n-1$, 
\begin{align*}
N'(K(n,r))&=N(K(n,r))=M(K(n,r))=M'(K(n,r))=\left\lceil \frac{n-1}{n-r}\right\rceil,\\ 
N'(K'(n,r))&=N(K'(n,r))=M(K'(n,r))=M'(K(n,r))=\left\lceil \frac{n}{n-r}\right\rceil,\\
N'(L(n,r))&=N(L(n,r))=M(L(n,r))=M'(L(n,r))=\left\lceil \frac{n}{n-r}\right\rceil.
\end{align*}
\end{prop}
\begin{proof}
To see that the semigroups $K(n,r)$ and $K'(n,r)$ are generated by their maximum $\mathcal{J}$-classes see \cite{Howie&McFadden:1990,Garba:1990}, respectively; and by Lemma \ref{L(n,r)}, this assertion is true for $L(n,r)$. Hence, by Lemma \ref{minimal-generating-set} every minimal generating set for these semigroups is contained in their maximum $\mathcal{J}$-classes. On the other hand, the rank of elements in the maximum $\mathcal{J}$-class for these semigroups is $r$. Hence, Theorem \ref{lower-bound-N'} implies that
  \begin{align*}
N'(K(n,r)) \geq \left\lceil \frac{n-1}{n-r}\right\rceil,\\ 
N'(K'(n,r)) \geq \left\lceil \frac{n}{n-r} \right \rceil,\\
N'(L(n,r)) \geq \left\lceil \frac{n}{n-r} \right \rceil.
\end{align*}
The proof is completed by showing that
\begin{align*}
M'(K(n,r)) \leq \left\lceil \frac{n-1}{n-r}\right\rceil,\\ 
M'(K'(n,r)) \leq \left\lceil \frac{n}{n-r} \right \rceil,\\
M'(L(n,r)) \leq \left\lceil \frac{n}{n-r} \right \rceil.
\end{align*}
First, we prove that 
$ M'(K(n,r)) \leq \left\lceil \frac{n-1}{n-r}\right\rceil.$ Let $A$ be a minimal generating set of $K(n,r)$. We show that there exists some product of at most $\left\lceil \frac{n-1}{n-r}\right\rceil$ generators in $A$ which is a constant transformation. Denote by $J$ the maximum $\mathcal{J}$-class of $K(n,r)$. By Lemma \ref{maximum-j-class}, $A$ covers the $\mathcal{L}$-classes of $J$ and the $\mathcal{R}$-classes of $J$. Since $A$  covers the $\mathcal{L}$-classes of $J$, there exists a transformation $f_1 \in A$ such that $\Im(f_1)=\{1,2,\ldots,r\}$. Since $A$ also covers the $\mathcal{R}$-classes of $J$, we can define $f_2,f_3,\ldots, f_{\ell} \in A$ as follows:
	for $i \geq 2$, if $\rank(f_1f_2\cdots f_{i-1}) > n-r+1$, then   
 choose  $f_i \in A$ that collapses $n-r+1$ elements in the image of $f_1f_2\cdots f_{i-1}$; otherwise, choose  $f_i \in A$  that collapses all the elements in the image of $f_1f_2\cdots f_{i-1}$. It is enough to check that $f_1f_2\cdots f_{\left\lceil \frac{n-1}{n-r}\right\rceil}$ is a constant transformation. For $r=1$, this is trivial. Let $r \geq 2$. If $r \leq n-r+1$, then $f_1f_2$ is a constant transformation. On the other hand, the inequalities $2 \leq r \leq n-r+1$ imply $ 2=\left\lceil \frac{n-1}{n-r}\right\rceil$. Suppose next that $r > n-r+1$. There exists $k\geq 2$ such that $ \rank(f_1f_2\cdots f_k) \leq n-r+1$ and $ \rank(f_1f_2\cdots f_{k-1}) > n-r+1$. Since $f_{k+1}$ collapses all the elements in the image of $f_1f_2 \cdots f_k$, then $f_1f_2\cdots f_{k+1}$ is a constant transformation. It remains to show that $k+1=\left\lceil \frac{n-1}{n-r}\right\rceil$. Note that 
$$\rank(f_1f_2 \cdots f_i)=r-(i-1)(n-r), \mbox{ for } 1 \leq i \leq k.$$ 
Hence, we have
\begin{equation}\label{3.3}
\rank(f_1f_2\cdots f_k)=r-(k-1)(n-r) \leq n-r+1,
\end{equation}
and
\begin{equation}\label{3.4}
\rank(f_1f_2\cdots f_{k-1})=r-(k-2)(n-r) > n-r+1.
\end{equation}
The inequalities \eqref{3.3} and \eqref{3.4} imply that 
$$ k < \frac{n-1}{n-r} \leq k+1,$$
which is the desired conclusion.

Next, we prove that
$$M'(L(n,r)) \leq \left\lceil \frac{n}{n-r}\right\rceil.$$ Let $B$ be a minimal generating set of $L(n,r)$. We show that there exists some product of at most $\left\lceil \frac{n}{n-r}\right\rceil$ generators in $B$ which is the empty transformation. By Lemma \ref{maximum-j-class}, $B$ covers the $\mathcal{R}$-classes of $J_r$. Hence, there exists a transformation $g_1 \in B$ such that $ 1,2,\ldots,n-r \not \in \dom(g_1)$. We can define $g_2,g_3,\ldots, g_{\ell} \in B$ as follows:
	for $i \geq 2$, if $\rank(g_1g_2\cdots g_{i-1}) \geq n-r+1$   
 choose  $g_i \in B$ such that $n-r$ elements in the image of $g_1g_2\cdots g_{i-1}$ are excluded from $\dom(g_i)$; otherwise, choose  $g_i \in A$  such that all elements in the image of $g_1g_2\cdots g_{i-1}$ are excluded from $\dom(g_i)$. It is enough to check that $g_1g_2\cdots g_{\left\lceil \frac{n}{n-r}\right\rceil}$ is the empty  transformation. If $r=1$, then $g_1g_2$ is the empty transformation and $\left\lceil \frac{n}{n-1}\right\rceil=2$. Let $r \geq 2$. If $ r < n-r+1 $, then $g_1g_2$ is the empty transformation. On the other hand the inequalities  $2 \leq r <n-r+1 $ imply $\left\lceil \frac{n}{n-r}\right\rceil=2$. Suppose next that $r \geq n-r+1$. There exists $k \geq 2 $ such that 
 \begin{align}
 0 <\rank(g_1g_2\ldots g_k) < n-r+1,\label{eq:rank}\\ 
  \rank(g_1g_2\cdots g_{k-1} ) \geq n-r+1.\label{eq:rank1}
  \end{align}
 Since none of the elements in the image of $g_1g_2\ldots g_k$ is in the domain of $g_{k+1}$, then $g_1g_2\cdots g_{k+1}$ is the empty transformation. It remains to show that $k+1=\left\lceil \frac{n}{n-r}\right\rceil$.
 By definition of  $g_k$, we have 
\begin{equation}\label{equation1}
\rank(g_1g_2\cdots g_k)=n-k(n-r),
\end{equation}
and 
\begin{equation}\label{equation2}
\rank(g_1g_2\cdots g_{k-1})=n-(k-1)(n-r).
\end{equation}
Substituting \eqref{equation1} in \eqref{eq:rank}  and \eqref{equation2} in \eqref{eq:rank1}, we obtain
$$ k < \frac{n}{n-r} \leq k+1,$$
which is the desired conclusion.

Finally, we consider the semigroup $K'(n,r)$. Let $C$ be a minimal generating set of $K'(n,r)$.  By Lemma \ref{maximum-j-class}, $C$ covers the $\mathcal{R}$-classes of the maximum $\mathcal{J}$-class of $K'(n,r)$. On the other hand, the maximum $\mathcal{J}$ -class of $L(n,r)$ is contained in the maximum $\mathcal{J}$-class of $K'(n,r)$.  Then by Lemma \ref{partial-permutation}, we may choose 
$g_1,g_2,\ldots, g_{\left\lceil \frac{n}{n-r}\right\rceil} \in C$. This shows that 
$N(K'(n,r),C) \leq \left\lceil \frac{n}{n-r}\right\rceil$ and so 
$M'(K'(n,r)) \leq \left\lceil \frac{n}{n-r}\right\rceil$.   
\end{proof}

In the sequel, we try to calculate the maximum $A$-depth over all minimal generating sets. We just   
  apply the following simple lemma to establish an upper bound for $M'(S)$ provided that $S$ is a semigroup generated by the maximal $\mathcal{J}$-classes.  First, we need to introduce some notation.
\begin{notation}
Let $S$ be a finite semigroup. Denote by $J_M$ \index[notation]{$J_M$} the set of all the maximal $\mathcal{J}$-classes of $S$. For every $\mathcal{J}$-class $J$  of $S$ denote by
$h_J$\index[notation]{$h_J$}, $\ell_J$\index[notation]{$\ell_J$} and $r_J$\index[notation]{$r_J$} the number of classes in $\mathcal{J}$  for the relations $\mathcal{H}$,
  $\mathcal{L}$ and $\mathcal{R}$, respectively.
\end{notation}
\begin{lem}\label{length-in-Jclass}
Let  $J$ be a maximal $\mathcal{J}$-class of a semigroup $S$. Let $A$ be a generating set of $S$. The length of elements in $J$ with respect to $A$ is at most $\min \{\ell_Jh_J,r_Jh_J\}$. 
\end{lem}

\begin{proof}
Let $x \in J$ and  $l_A(x)=k$. There exist $a_1,a_2, \ldots, a_k \in A \cap J$ such that $x=a_1a_2\ldots a_k$. Since, $a_1, a_1a_2, \ldots , a_1a_2 \ldots a_k$ are $k$ distinct elements in the same $\mathcal{R}$-class, then $k \leq \ell_J h_J$. On the other hand, $a_k, a_{k-1}a_k, \ldots, a_1a_2 \ldots a_k$ are $k$ distinct elements in the same $\mathcal{L}$-class, then $k \leq r_J h_J$. Hence, we have $$k \leq \min \{\ell_Jh_J,r_Jh_J\}.\qedhere$$
\end{proof}

\begin{prop}\label{upper-bound-M'}
Let $S$ be a finite semigroup. If $S$ is  generated by the maximal $\mathcal{J}$-classes, then 
$$M'(S) \leq N(S, \cup_{J \in J_M} J) \max_{J \in J_M}\min \{\ell_Jh_J,r_Jh_J\}.$$
\end{prop}

\begin{proof}
Let $A$ be a minimal generating set of $S$. It suffices to show that $N(S,A)$ is bounded above by the proposed bound.
Let $N(S,\cup_{J \in J_M} J)=k$. There exist $x_1,x_2,\ldots,x_k \in \cup_{J \in J_M} J$ such that $x=x_1 x_2 \ldots x_k \in \ker(S)$.  We have
$l_A(x) \leq \sum_{i=1}^k { l_A (x_i)}$. 
According to 
Lemma \ref{length-in-Jclass},  $l_A(x_i) \leq \min \{\ell_Jh_J,r_Jh_J\}$ 
for some maximal $\mathcal{J}$-class of $S$ containing $x_i$. If $M$ is the maximum of $\min \{\ell_J h_J,r_J h_J\}$ over all maximal $\mathcal{J}$-classes of $S$, then  $l_A(x_i) \leq M$ for $1 \leq i \leq k$. This shows that $l_A(x) \leq kM$. Hence $N(S,A) \leq kM$ which is the desired conclusion.
\end{proof}
\section{$A$-depth and products of semigroups}\label{se:product}

We did some attempts to understand the behavior of the depth parameters with respect to products (direct product and wreath product) of semigroups. Here we deal mostly with monoids rather than semigroups because it is easier to say something about minimal generating sets when the components of the product are two monoids.  
  
  \subsection{ Direct product}
  Let $S, T$ be two finite monoids. We are interested in estimating the parameters
  \begin{align*}
  N'(S \times T),~
  N(S \times T),
  \end{align*}
  with respect to the corresponding parameters for $S$ and $T$. 
  First, we observe that the kernel of the direct product of two finite semigroups is the product of the kernels of its components.
  \begin{lem}
  Let $S$, $T$ be two finite semigroups. Then
  $$\ker(S \times T) = \ker(S) \times \ker(T).$$
  \end{lem}
  \begin{proof}
  It is easy to see that $\ker(S) \times \ker(T)$ is an ideal of $S \times T$. Since $\ker(S \times T)$ is the minimum ideal of $S \times T$, then $\ker(S \times T) \subseteq \ker(S) \times \ker(T).$ It remains to show that $\ker(S) \times \ker(T)$ is just one $\mathcal{J}$-class. It follows from the fact that the direct product of two simple semigroups is a simple semigroup; it is easy to justify this fact by considering that a semigroup $S$ is simple if and only if $SaS=S$ for every $a \in S$ \cite{Pin:1986}.
  \end{proof}
  Next, we need to establish a relationship between generating sets of the direct product and generating sets of its components. We could not find a nice general method for constructing a generating set of minimum size for $S_1 \times S_2$ when the semigroups $S_1,S_2$ do not contain an identity element. Just as an easy example we consider the product of two monogenic semigroups.
    
\begin{ex}
 Let $i,n,j,m \geq 1$. Then the depth parameters are all
equal for $C_{i,n} \times C_{j,m}$ and they are given by
the formula 
$$N(C_{i,n} \times C_{j,m})=\left\{ \begin{array}{ccccc}
0  & \mbox{if} & i=j=1\\
i & \mbox{if} & j=1, i \not =1\\
j & \mbox{if} & i=1, j \not =1\\
2 & \mbox{if} & i,j \not = 1\\
\end{array}\right.$$
Furthermore, if $i \not =1$, or $j \not = 1$, then $C_{i,n} \times C_{j,m}$ has a unique minimal generating set.
\end{ex}
\begin{proof}
Let $C_{i,n}=\langle a:~ a^{i+n}=a^i \rangle$
and $C_{j,m}=\langle b:~ b^{j+m}=b^j \rangle.$ In case both $i,j$ are equal to $1$, these cyclic semigroups are
groups and, therefore, so is
their product. Because $N(G)=0$ for any group $G$ then $N(C_{1,n} \times C_{1,m})=0$. 
If $j=1 , i \not =1$, then the maximum $\mathcal{J}$-class of $C_{i,n}\times C_{1,m}$ is $\{a\} \times C_{1,m}$. If $A$ is any generating set of  $C_{i,n}\times C_{1,m}$ then $A$ must contain $\{a\} \times C_{1,m}$ because $a$ can not be written as a product of two elements. On the other hand, $\{a\} \times C_{1,m}$ generates $C_{i,n}\times C_{1,m}$ because, if $(a^k,b^l)\in C_{i,n}\times C_{1,m}~ \mbox{for some}~ k>1,$ then $ (a^k,b^l)=(a,b^l)(a,1)^{k-1}$. Therefore, $\{a\} \times C_{1,m}$ is the unique generating set of $C_{i,n}\times C_{1,m}$ of minimum size and $$\ker(C_{i,n}\times C_{1,m})=\{a^i,a^{i+1},\ldots ,a^{i+n}\} \times C_{1,m}.$$ Note that $(a,1)^i \in \ker(C_{i,n}\times C_{1,m})$ and, because the first component of every element in the generating set is $a$, the product of generators with less than $i$ factors can not reach the minimum ideal. Therefore $N(C_{i,n} \times C_{1,m})=i$. The case where $i=1 ,j\not = 1$ is similar.
Now, let $i,j\not =1$. 
  We show that 
  $$A=\{(a,b^k)| 1 \leq k \leq j+m-1 \} \cup \{(a^l,b)| 1 \leq l \leq i+n-1\}$$ is the unique minimal generating set of $C_{i,n} \times C_{j,m}$. Every generating set must contain $A$ because $a$ and $b$ cannot be written as products of any other elements. Furthermore, if $ (a^s,b^t) \in C_{i,n} \times C_{j,m} ~\mbox{for some} ~s,t > 1$ then $ (a^s,b^t)=(a,b^{t-1})(a^{s-1},b)$. Hence, $A$ generates $C_{i,n} \times C_{j,m}$. 
We have $a^i \in \ker(C_{i,n})$ and $a^j \in \ker(C_{j,m})$. In view of Lemma 11, it
 follows that $(a,b^{j-1})(a^{i-1},b)=(a^i,b^j) \in \ker(C_{i,n} \times C_{j,m})$.
This proves that $N (C_{i,n} \times C_{j,m})=2$.
\end{proof}

 In the next example, we treat the case where just one of the components in the direct product is a cyclic semigroup.
 
 \begin{ex}
 Let $S$ be a semigroup and let $i > 1,~n \geq 1$. Then, the following inequality holds: $$M'(S\times C_{i,n}) \leq i.$$
 \end{ex}
 \begin{proof}
 Let $$C_{i,n}=\{a,a^2,\ldots,a^i,a^{i+1},\ldots,a^{n+i-1}\}.$$ If $A$ is any generating set of $S \times C_{i,n}$ then $S\times \{a\} \subseteq A$. Let $x \in \ker(S)$. We have $(x,a) \in A$ and $(x,a)^i=(x^i,a^i) \in \ker(S)\times \ker(C_{i,n})$, whence $N(S\times C_{i,n},A) \leq i$.   
 \end{proof}
 From now on, we consider monoids rather than semigroups. Let $A_1,A_2$ be two minimal generating sets of the monoids $M_1 \neq \{1\}$ and $M_2 \neq \{1\}$, respectively.  If  $ (1,1) \not \in (A_1 \times \{1\} )\cup (\{1\} \times A_2)$, then $A=(A_1 \times \{1\} )\cup (\{1\} \times A_2) $ is a minimal generating set of $M_1 \times M_2$; otherwise $A=(A_1 \times \{1\} )\cup (\{1\} \times A_2) \setminus \{(1,1)\} $ is a minimal generating set of $M_1 \times M_2$. Let $N'(M_1)=t_1,~ N'(M_2)=t_2$. There exist $a_1,a_2,\ldots,a_{t_1} \in A_1 \setminus \{1\}, ~a'_1,a'_2,\ldots,a'_{t_2} \in A_2 \setminus \{1\}$ such that $a_1a_2\ldots a_{t_1} \in \ker(M_1),~a'_1a'_2\ldots a'_{t_2} \in \ker(M_2)$. So, we have $(a_1a_2\ldots a_{t_1},a'_1a'_2 \ldots a'_{t_2} ) \in \ker(M_1 \times M_2)$. On the other hand, the length of $(a_1a_2\ldots a_{t_1},a'_1a'_2 \ldots a'_{t_2} )$ with respect to $A$ is
 $t_1+t_2$. It follows that 
 \begin{equation}\label{N'}
 N'(M_1 \times M_2 ) \leq N'(M_1) + N'(M_2).
 \end{equation}
 It is natural to ask whether there is an expression like inequality \eqref{N'} for the other parameters $N,M,M'$.
 In fact, if $A$  or $A \setminus \{(1,1)\}$ is a generating set of minimum size then we could derive a similar inequality for $N$. But $A$ may not be a generating set of minimum size. 
In general, we may establish the following lemma concerning the rank of the direct product of two finite monoids.
\begin{defi}
For a finite monoid $M$ with group of units $U$, the rank of $M$ modulo $U$ is the minimum number of elements in $M \setminus U$ which together with $U$ generate $M$.
\end{defi}
  \begin{lem}\label{lem:generating-set-direct-product}
  Let $M_1,M_2$ be two finite monoids. Denote by $U_{i}$ the group of units of $M_i$ and by $k_i$ the rank of $M_i$ modulo $U_{i}$. Let $A'_i \subseteq M_i \setminus U_i$ 
   be such that 
$|A'_i|=k_i~\mbox{and}~M_i = \langle U_i \cup A'_i\rangle.$
   Let $B$ be a generating set of $U_1 \times U_2$. Then the set 
   \begin{equation*}
   C=B \cup (A'_1 \times \{1\}) \cup (\{1\} \times A'_2), 
\end{equation*}    
  is a generating set of $M_1 \times M_2$. Furthermore, we have
  $$\rank(M_1 \times M_2)= \rank(U_1 \times U_2 )  + k_1+k_2.$$
  \end{lem} 
  \begin{proof}
  Let $(x,y) \in M_1 \times M_2$. We show that $(x,y) \in \left \langle C \right\rangle$. It is enough to show that $(x,1),(1,y) \in \left \langle C \right\rangle$.
  We know that $x$ is a product of elements in $U_1 \cup A'_1$. Let $x=x_1x_2\ldots x_t$ for some $x_i \in U_1 \cup A'_1$. Hence, we have $(x,1)= \prod_{i=1}^t (x_i,1)$. For $1 \leq i \leq t$; if $x_i \in A'_1$ then we have $(x_i,1) \in C$; if $x_i \in U_1$ then we have $(x_i,1) \in U_1 \times U_2 = \langle B \rangle$. Thus, $(x_i,1) \in \langle C \rangle$, which implies that $(x,1) \in \langle C \rangle$. In the same manner, we can see that $(1,y) \in \langle C \rangle$.  
 
The rest of the proof consists in showing that $C$ is a generating set of minimum size when $B$ is a generating set of minimum size or $U_1 \times U_2$. Let $X$ be a generating set of $M_1 \times M_2$. Write $\bar{M_1}=M_1 \setminus U_1$ and 
$\bar{M_2}=M_2 \setminus U_2$. We have
\begin{equation}\label{generating-set-direct-product}
M_1 \times M_2= (U_1 \times U_2) \cup (U_1 \times \bar{M_2}) \cup (\bar{M_1} \times U_2) \cup (\bar{M_1} \times \bar{M_2}).
\end{equation}
It is clear that $X$ has at least $\rank(U_1 \times U_2)$ elements in $U_1 \times U_2$. Furthermore, $(U_1 \times \bar{M_2}) \cup (\bar{M_1} \times \bar{M_2})$ and $ (\bar{M_1} \times U_2) \cup (\bar{M_1} \times \bar{M_2})$ are ideals of $M_1 \times M_2$, then $X$ has at least $k_1$ elements in $\bar{M_1} \times U_2$ and $k_2$ elements in $U_1 \times \bar{M_2}.$ These facts combined with the pairwise disjointness of the subsets in the right side of \eqref{generating-set-direct-product} gives  $|X| \geq \rank(U_1 \times U_2) + k_1+ k_2,$ which completes the proof.
 \end{proof}
 \begin{rem}\label{re:generating-set-minimum-size}
 Let $A_1, A_2$ be two generating sets of $M_1,M_2$ with minimum size. If $(1,1) \not \in (\{1\} \times A_2) \cup (A_1 \times \{1\})$ , then the size of  the generating set $A=(\{1\} \times A_2) \cup (A_1 \times \{1\})$ is equal to $ \rank(M_1) + \rank(M_2)= \rank(U_1)+k_1+ \rank(U_2) + k_2$, where $k_i$ is the rank of $M_i$ modulo $U_i$. Therefore, by Lemma \ref{lem:generating-set-direct-product}, if $\rank (U_1 \times U_2)= \rank(U_1)+ \rank(U_2)$, then 
 the generating set $A$ is a generating set of minimum size. On the other hand,  by the minimality of $A_1$ and $A_2$, $(1,1) \in A=(\{1\} \times A_2) \cup (A_1 \times \{1\})$ if and only if $U_1=U_2=\{1\}$. Hence,  if $(1,1) \in A$ then $|A\setminus\{(1,1)\}|= \rank(U_1)+k_1+ \rank(U_2) + k_2-1=k_1+k_2+1$. But also by Lemma \ref{lem:generating-set-direct-product}, $\rank(M_1 \times M_2)= k_1+k_2+1$. So, $A \setminus \{(1,1)\}$ is a generating set of minimum size of $M_1 \times M_2$.
 \end{rem}

\begin{theorem}\label{th:upper-bound-direct-product}
Let $M_1$ and $M_2$ be two finite monoids. Then, we have
$$N(M_1 \times M_2) \leq (N(M_1)+N(M_2))D(U_1 \times U_2),$$ provided that $D(U_1 \times U_2) \not = 0$.
Furthermore, if $\rank(U_1 \times U_2 )= \rank(U_1)+ \rank(U_2)$ (and also in the case $D(U_1 \times U_2)  = 0$) then we have
$$N(M_1 \times M_2) \leq N(M_1)+N(M_2).$$
\end{theorem}
\begin{proof}
 Let $A_1, A_2$ be generating sets of minimum size of $M_1,M_2$, respectively, such that $N(M_1,A_1)=N(M_1)$ and $N(M_2,A_2)=N(M_2)$. Let $B$ be a generating set of $U_1 \times U_2$ of minimum size. Let 
 \begin{equation*}
   C=B \cup (A'_1 \times \{1\} )  \cup (\{1 \} \times A'_2) , 
\end{equation*}
where $A'_i=A_i \setminus U_i$.   
  There exist 
 $x_1,x_2,\ldots,x_{N(M_1)} \in A_1$ and $y_1,y_2,\ldots,y_{N(M_2)} \in A_2$ such that $x_1x_2\ldots x_ {N(M_1)} \in \ker(M_1)$ and $y_1y_2\ldots y_ {N(M_2)} \in \ker(M_2)$. Hence, the pair $(x_1x_2\ldots x_ {N(M_1)},y_1y_2\ldots y_ {N(M_2)})$ belongs to $\ker(M_1 \times M_2)$. The following equality
 $$(x_1x_2\ldots x_ {N(M_1)},y_1y_2\ldots y_ {N(M_2)})=\prod_{i=1}^{N(M_1)}(x_i,1) \prod_{j=1}^{N(M_2)}(1,y_j),$$ 
implies that
$$l_C((x_1x_2\ldots x_ {N(M_1)},y_1y_2\ldots y_ {N(M_2)}))  \leq \sum_{i=1}^{N(M_1)} l_C(x_i,1) + \sum_{j=1}^{N(M_2)} l_C(1,y_j).$$  
For $1 \leq i \leq N(M_1)$, if $x_i \in A'_1$ then we have $l_C(x_i,1)=1$; otherwise, we have $l_C(x_i,1)\leq \diam(U_1 \times U_2,B)$. For $1 \leq i \leq N(M_2)$, if $y_i \in A'_2$ then we have $l_C(1,y_i)=1$; otherwise, we have $l_C(1,y_i)\leq \diam(U_1 \times U_2,B)$. Let
$$s_1=|\{x_1,x_2,\ldots,x_{N(M_1)}\} \cap A'_1|,$$ and $$s_2=|\{y_1,y_2,\ldots,y_{N(M_2)}\} \cap A'_2|.$$ Then the length of $(x_1x_2\ldots x_ {N(M_1)},y_1y_2\ldots x'_ {N(M_2)})$, in the generating set $C$, is at most
\begin{align}\label{M1xM2}
&s_1+s_2+(N(M_1)+N(M_2)- (s_1+s_2))\diam(U_1 \times U_2,B)\\ 
=&(N(M_1)+N(M_2))\diam(U_1 \times U_2,B) \nonumber \\
+&(1-\diam(U_1 \times U_2,B))(s_1+s_2)\nonumber.
\end{align}
 The upper bound in \eqref{M1xM2} depend on the integers $s_1,s_2$ and the generating set $B$. Now we try to remove these parameters from the proposed upper bound. Since  
$1-\diam(U_1 \times U_2,B) \leq 0$ and $s_1+s_2 \geq 0$ then 
\begin{align}
&(N(M_1)+N(M_2))\diam(U_1 \times U_2,B) \nonumber \\+
&(1-\diam(U_1 \times U_2,B))(s_1+s_2)\nonumber \\
\leq & (N(M_1)+N(M_2))\diam(U_1 \times U_2,B) .\label{dia}
\end{align} 
Substituting $D(U_1 \times U_2)$ for $\diam(U_1 \times U_2,B)$ in \eqref{dia}
establishes the first statement of the theorem.

Now we prove the second statement. Let $\rank(U_1 \times U_2)= \rank(U_1)+\rank(U_2$. According to Remark \ref{re:generating-set-minimum-size}, the set $A=(\{1 \} \times A_2) \cup (A_1 \times \{1\})$ is a generating set of $M_1 \times M_2$ of minimum size. Suppose that $N(M_1,A_1)=N(M_1)=t_1$ and $N(M_2,A_2)=N(M_2)=t_2$. There exist $a_1,a_2,\ldots,a_{t_1} \in A_1, ~a'_1,a'_2,\ldots,a'_{t_2} \in A_2$ such that $a_1a_2\ldots a_{t_1} \in \ker(M_1),~a'_1a'_2\ldots a'_{t_2} \in \ker(M_2)$. So, we have $(a_1a_2\ldots a_{t_1},a'_1a'_2 \ldots a'_{t_2} ) \in \ker(M_1 \times M_2)$. On the other hand, the length of $(a_1a_2\ldots a_{t_1},a'_1a'_2 \ldots a'_{t_2} )$ with respect to $A$ is at most $t_1+t_2$. It follows that 
 \begin{equation}\label{N}
 N(M_1 \times M_2 ) \leq N(M_1) + N(M_2),
 \end{equation}
 which is the desired conclusion. 
 For the case that $D(U_1 \times U_2)=0$ we have $U_1\times U_2=U_1=U_2=\{1\}$. According to Remark \ref{re:generating-set-minimum-size}, the set $A=(\{1 \} \times A_2) \cup (A_1 \times \{1\})\setminus \{(1,1)\}$ is a generating set of $M_1 \times M_2$ of minimum size.  Suppose that $N(M_1,A_1)=N(M_1)=t_1$ and $N(M_2,A_2)=N(M_2)=t_2$. There exist $a_1,a_2,\ldots,a_{t_1} \in A_1\setminus \{1\}, ~a'_1,a'_2,\ldots,a'_{t_2} \in A_2\setminus \{1\}$ such that $a_1a_2\ldots a_{t_1} \in \ker(M_1),~a'_1a'_2\ldots a'_{t_2} \in \ker(M_2)$. So, we have $(a_1a_2\ldots a_{t_1},a'_1a'_2 \ldots a'_{t_2} ) \in \ker(M_1 \times M_2)$. On the other hand, the length of $(a_1a_2\ldots a_{t_1},a'_1a'_2 \ldots a'_{t_2} )$ with respect to  $A\setminus \{(1,1)\}$ is at most $t_1+t_2$. It follows that 
 \begin{equation}\label{N}
 N(M_1 \times M_2 ) \leq N(M_1) + N(M_2),
 \end{equation}
 which is the desired conclusion. 
  
\end{proof} 

The remainder of this section is devoted to the computation of $N(T_n \times T_m)$ for $n,m \geq 3$.
\begin{lem}\label{Sn-generating-set}
 For $n \geq 3$ the symmetric group ${S}_n$ can be generated by two elements of coprime order. 
 \end{lem}
 \begin{proof}
  Define the permutations  $a,a'$ and $b$ as follows:
$$ a=  \begin{pmatrix}
1 & 2 & 3 & \ldots & n\\
2 & 3 & 4 & \ldots & 1
\end{pmatrix}, 
 a'=  \begin{pmatrix}
1 & 2 & 3 & \ldots & n\\
1 & 3 & 4 & \ldots & 2
\end{pmatrix}
 $$ and
$$b=  \begin{pmatrix}
1 & 2 & 3 & \ldots & n\\
2 & 1 & 3 & \ldots & n
\end{pmatrix}.$$ 
It is known that the full cycle $a$ and the transposition $b$ generate $S_n$ \cite{Hungerford:1980}. On the other hand, note that
$a'b = a$. Hence, the sets $\{a,b\}$ and $\{a',b\}$ are generating sets of $S_n$. Note that 
$$ \ord(a)=n,~  \ord(b)=2,~  \ord(a')=n-1.$$  Therefore, for odd $n$, the set $A=\{a,b\}$ and, for even $n$, the set $A'=\{a',b\}$ are the desired generating sets.
\end{proof}
 For $n,m \geq 3$, let $U_1$ and $U_2$ be the group of units of $T_n$ and $T_m$, respectively.
We show that $N(T_n \times T_m) = N(T_n) \times N(T_m)$, while $U_1 \times U_2$ is neither trivial nor $\rank(U_1 \times U_2 )= \rank(U_1)+ \rank(U_2)$. More precisely, we have $U_1=S_n$ and $U_2=S_m$. Let $S_n=\langle a,b \rangle$ and $S_m= \langle c,d \rangle$ such that both of the pairs $a,c$ and $b,d$  are of coprime orders (see Lemma \ref{Sn-generating-set}). We show that $S_n \times S_m = \langle (a,c),(b,d) \rangle.$ It is enough to show that $(a,1),(b,1),(1,c),(1,d) \in \langle (a,c),(b,d) \rangle.$  This is because $a,c$ and $b,d$ are of coprime orders. In fact, if $x,y$ are of coprime order then there exists a power of $(x,y)$ which is equal to $(x,1)$ and there exists a power of $(x,y)$ which is equal to $(1,y)$.  Hence, we have $\rank(S_n \times S_m)=2$, which is not equal to $\rank(S_n) + \rank (S_m).$
  \begin{lem}
  \label{rank-lower-bound2}
  Let $ S= \{ f \in T_n | ~\rank(f) \geq n-1 \}$. If $ \rank (f_1 f_2
  \ldots f_k) = 1$ for some $f_1,f_2,\ldots,f_k \in S$ then at least
  $n-1$ elements of $f_1,f_2,\ldots,f_k$ are of $\rank ~n-1$.
\end{lem}

\begin{proof}
  For every $f,g \in T_n$, if $\rank(f)=n$ then
  $\rank(fg)=\rank(gf)=\rank(g)$. Thus, without loss of generality, we can
  suppose that all the $f_i$ have rank $n-1$ and apply Lemma \ref{rank-lower-bound1}.
\end{proof}

\begin{lem}\label{generating set}
  Let $n,m \geq 2$. Let $A$ be a generating set of ${S}_n \times {S}_m$ of minimum size and
  $a \in T_n$ be a function of rank $n-1$, $b \in T_m$ be a function
  of rank $m-1$. Then $B=A \cup \{(a,a')\} \cup \{(b',b)\}$, where
  $(a',b')\in {S}_m \times {S}_n$, is a generating set of $T_n \times T_m$
  of minimum size. Furthermore, all generating sets of $T_n \times T_m$
  of minimum size are of this form.
\end{lem}

\begin{proof}
  First we show that $B$ generates $T_n \times T_m$. Since
  $$(a,1)=(a,a')(1,a'^{-1}) \quad \mbox{and} \quad (1,b)=(b',b)(b'^{-1},1),$$
 $B$ generates $(a,1),(1,b)$. Let $(f,g)\in T_n \times T_m$.
  Because $f \in T_n$, there exist $f_1,f_2,\ldots,f_k \in {S}_n
  \cup \{a\}$ such that $f=f_1 f_2 \ldots f_k$. Because $g \in T_m$, there exist $g_1,g_2,\ldots,g_l \in {S}_m \cup \{b\}$ such that
  $g=g_1 g_2 \ldots g_l$. Then, we have
  $$ (f,g)=(f_1,1)(f_2,1) \ldots (f_k,1)(1,g_1)(1,g_2) \ldots (1,g_l).$$
  Every $(f_i,1)$ either is $(a,1)$ or belongs to ${S}_n \times {S}_m$
  and every $(1,g_i)$ either is $(1,b)$ or belongs to ${S}_n \times
  {S}_m$. Therefore ,$B$ generates $(f_i,1), (1,g_j) $ for
  $i=1,2,\ldots,k$ and $j=1,2,\ldots,l$. Consequently, $B$ generates
  $(f,g)$.

  Let $C$ be a generating set of $T_n \times T_m$ of minimum size.
  Then, $C$ must contain a generating set of the maximum
  $\mathcal{J}$-class which is ${S}_n \times {S}_m$. On the other hand, the
  maximum $\mathcal{J}$-class ${S}_n \times {S}_m$ is a subsemigroup; hence,
  one cannot obtain any elements in the $\mathcal{J}$-classes below by
  multiplying just elements on the maximum $\mathcal{J}$-class.
  Therefore, $C$ must contain some elements of some $\mathcal{J}$-classes
  below the maximum $\mathcal{J}$-class. There are exactly two
  $\mathcal{J}$-classes which are below the maximum
  $\mathcal{J}$-class and above all other $\mathcal{J}$-classes.
  Therefore, $C$ must intersect each of them in at least one element.
  Note that all such elements have the respective forms $(a,a')$ and
  $(b',b)$ as described in the statement of the lemma. This shows that
  $A \cup \{(a,a')\} \cup \{(b',b)\}$ is a generating set of minimum
  size and all generating sets of minimum size are of this form.
\end{proof}

\begin{prop}
  \label{number}
	If $T_n,T_m $ are two full transformation semigroups, then 
  $$N(T_n \times T_m)= m+n-2.$$
\end{prop}

\begin{proof}If $n=m=1$ then we have $N(T_1\times T_1)=0=1+1-2$. If $n=1$ or $m=1$ then the equality  holds by Corollary \ref{N-transformation-cor1}. 
   Suppose that $n,m \geq2$. Let $A$ be a generating set of ${S}_n \times {S}_m$ of minimum size.
  Consider functions $\alpha,\beta$ defined by 
  \begin{align*}
    \alpha&=\left( 
      \begin{array}{cccccc}
        1 & 2 & 3 & \ldots & n\\
        1 & 1 & 2 & \ldots & n-1\
      \end{array} \right),\\
    \beta&=\left( 
      \begin{array}{cccccc}
        1 & 2 & 3 & \ldots & m\\
        1 & 1 & 2 & \ldots & m-1\
      \end{array} \right).
  \end{align*}
  By Lemma \ref{generating set}, $ B= A \cup \{(\alpha,1)\} \cup \{(1,\beta)\}$ is a
  generating set of $T_n \times T_m$ of minimum size. We have
  $$ (\alpha,1)^{n-1} (1,\beta)^{m-1}
  =(\alpha^{n-1},1)\,(1,\beta^{m-1}) %
  =(\alpha^{n-1},\beta^{m-1}).$$
 Since the functions $\alpha^{n-1}$ and $\beta
  ^{m-1}$ are constant, we have $(\alpha,1)^{n-1}
  (1,\beta)^{m-1} \in \ker(T_n) \times \ker(T_m)$. This shows
  that $N(T_n \times T_m) \leq n-1 + m-1=m+n-2$.

  Next, we prove that $N(T_n \times T_m) \geq m+n-2$. Suppose
  $$ B=A \cup \{(a,a')\} \cup \{(b',b)\}$$
  is a generating set of $T_n \times T_m$ of minimum size and there are
  $$(f_1,g_1), (f_2,g_2), \ldots ,(f_k,g_k) \in B$$ such that 
  $$(f_1,g_1)(f_2,g_2)\ldots (f_k,g_k) \in \ker(T_n) \times \ker((T_m).$$
  Then $f_1 f_2 \ldots f_k \in \ker(T_n)$ and $g_1 g_2 \ldots g_k \in \ker(T_m)$.
  By Lemma \ref{rank-lower-bound2}, at least $n-1$ elements in $ \{
  f_1,f_2,\ldots, f_k \}$ are of rank $n-1$ and $m-1$ elements of $g_1, g_2, \ldots g_k$ are of rank $m-1$. Since every generator has at least one invertible component, the two conditions cannot be met by the same factor and therefore there are at least $m+n-2$ factors. 
\end{proof}
  
With the same argument, we can generalize Lemma \ref{generating set}
and Proposition \ref{number} to any finite product of full transformation
semigroups.
  
\begin{lem} 
  Let $A$ be a generating set of ${S}_{n_1} \times {S}_{n_2} \times \cdots
  \times {S}_{n_k}$ of minimum size and
  $$\alpha_t=(a_1,a_2,\ldots, a_t,\ldots, a_k)
  \in T_{n_1} \times T_{n_2} \times \cdots \times T_{n_k}
  \quad t=1,2,\ldots k$$
  such that 
  $$\rank(a_t)=n_t-1 ~\mbox{and} \quad a_i \in S_{n_i}
  \quad i \in \{1,2,\ldots,k\} \smallsetminus \{t\}.$$
  Then $B=A \cup (\bigcup_{t=1}^k \{\alpha_t\})$ is a generating set
  of $T_{n_1} \times T_{n_2} \times \cdots\times T_{n_k}$ of minimum size.
  Furthermore, all generating sets of $T_{n_1} \times T_{n_2} \times
  \cdots\times T_{n_k}$ of minimum size are of this form.
\end{lem}

\begin{prop}
 If $T_{n_i}$ for $1 \leq i \leq k$ are full transformation semigroups, then $$N(T_{n_1} \times T_{n_2} \times \cdots \times {T}_{n_k})
  = n_1+n_2+\cdots +n_k -k. $$ 
\end{prop}
\subsection{Wreath product}\label{su:wreath-product}
 By the prime decomposition theorem, every finite semigroup is a divisor of an iterated wreath product of its simple group divisors and the three-element monoid $U_2$ consisting of two right zeros and one identity element \cite{Rhodes&Steinberg:2009}. So we are looking for the analogues for the
wreath product of the results which we have obtained for the direct
product. We consider the \emph{wreath product} \index{product!wreath product} of transformation monoids as usual, that is  
$$(X,S)\wr (Y,T)=(X\times Y , S^Y\rtimes T),$$\index[notation]{$(X,S)\wr (Y,T)$} where the action defining the \emph{semidirect product} \index{product!semidirect product} is given by 
\begin{align*}
\quad T \times S^Y \rightarrow S^Y \\ 
(t,f)\mapsto \,^t\!f,\\\\
^t\!f:Y\rightarrow S \\ \index[notation]{$^t\!f$}
 y\mapsto (yt)f
\end{align*}
and the action of $S^Y \rtimes T$ on the set $X \times Y$ is described by 
$$ (x,y)(f,t)=(x(yf),yt).$$ 
 Note that we apply functions on the right. 
 Our aim is to give an upper bound for $N(S^Y \rtimes T)$\index[notation]{$S^Y \rtimes T$} in which $(X,S)$ and $(Y,T)$ are two transformation monoids and $S^Y \rtimes T$ is the semigroup of the wreath product  $(X,S) \wr (Y,T)$.
 Here, we introduce some notation which we use subsequently. For  $s \in S$ and $y \in Y$ let $(s)_y: Y \rightarrow S$\index[notation]{$(s)_y$} be the function defined by

$$
z(s)_y= \left \{
\begin{array}{ll}
s & \text{if} ~ z=y\\
1 & \text{otherwise}\\
\end{array} \right.
$$
and
for every $s \in S$ let $\bar{s}: Y \rightarrow S$ be the function defined by 
$y\bar{s}=s.$ \index[notation]{$\bar{1}$}

For a given monoid $S$ denote by $U_S$ \index[notation]{$U_S$} its \emph{group of units}\index{group of units}. We use the notation $\prod_{i=1}^n s_i$ \index[notation]{$\prod_{i=1}^n s_i$} for $s_1 s_2 \ldots s_n$ even in the case when the multiplication is not commutative.

\begin{lem}
\label{minimum ideal}
Let $(X,S)$ and $(Y,T)$ be two transformation monoids. The set
$$ E=\{(f,t): f \in \ker(S)^Y ,~ t \in \ker(T) , \text{ $f $ is a constant map} \} $$ is contained in the minimum ideal  of $S^Y \rtimes T$. 
\end{lem}

\begin{proof}
It is easy to check that every two elements in $E$ are
$\mathcal{J}$-related and $ \ker(S)^Y \times  \ker(T)$ is an ideal  of
$S^Y \rtimes T$. Hence, given $(f,t) \in E$ and $(g,t')\in \ker(S)^Y \times
\ker(T)$, it suffices  to show that there exist
$h,k \in S^Y , t_1,t_2 \in T$ such that 
$$(h,t_1)(g,t')(k,t_2)=(f,t).$$
 Since $t,t' \in \ker(T)$, there exist $t_1,t_2 \in \ker(T)$ such that
 $t_1t't_2 =t$. For each $s,s' \in \ker(S)$, there exist elements $h_{s,s'} ,k_{s,s'} \in \ker(S)$ such that $s'=h_{s,s'} s k_{s,s'}$. 
 Define the functions $h,k \in S^Y$  as follows: for each $y \in Y$, let
$$yh= h_{(yt_1)g,yf},$$
$$yk=\left\{ \begin{array}{ll}
k_{(xt_1)g,xf} & \text{if} ~y=xt_1t' ~ \mbox{for some } x \in Y,\\
1 & \text{otherwise}.\\
\end{array} \right.
$$
Note that the function $k$ is well-defined since, as $t_1$ and $t_1t'$ are in the same $\mathcal{R}$-class, the equality  
$\ker(t_1)=\ker(t_1t')$ holds. 
Now, we have
$$
(h,t_1)(g,t')(k,t_2)=(h\,^{t_1}\!g\,^{t_1t'}\!k,t_1t't_2)=(f,t)
$$
and the proof is complete.
\end{proof}
Note that by Lemma \ref{minimum ideal}, the following inequalities hold:
\begin{equation}\label{minimum-ideal-inequality}
E \subseteq \ker(S^Y \rtimes T) \subseteq \ker(S)^Y \times \ker(T).
\end{equation}
The following examples show that for some wreath products the inclusions in the inequalities \eqref{minimum-ideal-inequality} are proper and for the others are not.

In all the following examples, we consider the transformation semigroup $(Y,U_2)$ to be as following. Let $Y=\{1,2\}$ and $\alpha,\beta:Y \rightarrow Y$ be the constant functions $1,2$, respectively. Let $U_2=\{1,\alpha, \beta\}$. Then, $U_2$ acts faithfully on $Y$ and so $(Y,U_2)$ is a transformation semigroup.
\begin{ex}
 Let $(X,G)$ be a finite permutation group. Consider the wreath product $(X,G) \wr (Y,U_2)$. It is easy to see that the minimum ideal  of $G^Y \rtimes U_2$ is the whole  $\ker(G)^Y \times \ker(U_2)$. 
\end{ex}
\begin{ex}
Let $(X,T_3)$ be the full transformation semigroup of degree three. Consider the wreath product $(X,T_3) \wr (Y,U_2)$. Computer calculations give the minimum ideal  of $ T_3 ^Y \rtimes U_2$ to be the set 
$$E=\{(f,t): f \in \ker(T_3)^Y ,~ t \in U_{2} , \text{ $f $ is a constant map} \}.$$ 
\end{ex}
\begin{ex}
 Let $V$ be the transformation monoid generated by identity and two transformations 
\begin{equation}
a= \left(
\begin{array}{rrrrr}
1 &2 & 3 & 4 &5\\
1 & 4 & 1 & 4 & 1
\end{array} \right),
b= \left(
\begin{array}{rrrrr}
1 &2 & 3 & 4 &5\\
3 & 2 & 3 & 2 & 2
\end{array} \right).
\end{equation}
Computer calculations (using Mathematica) give 
the minimum ideal of $V^Y \rtimes U_2$ to have $16$ elements, while $E$ has $8$ elements and $\ker(V)^Y \times \ker(U_2)$ has $32$ elements. Hence, in this example the inequalities \eqref{minimum-ideal-inequality} are proper. 
\end{ex}
 \begin{lem} \label{rank-wreath-monoids}
Let $(X,S)$ and $(Y,T)$ be two transformation monoids. Then 
\begin{equation}\label{eq-rank-wreath}
\rank(S^Y \rtimes T) \geq \rank(S^Y \rtimes U_T) + \rank(T) - \rank(U_T).
\end{equation}
\end{lem}
\begin{proof}
Let $S_1=S^Y \rtimes U_T$ and $S_2= S^Y \rtimes (T \smallsetminus
U_T)$. It is easy to check that 
$S^Y \rtimes T =S_1 \cup S_2$ is a partition into two subsemigroups.
 Because $S_2$ is an ideal  of $S^Y \rtimes T$  
, every generating set of $S^Y \rtimes T$ must
contain a generating set of $S_1$. Moreover, we need at least
$\rank(T)-\rank(U_T)$ elements
for generating $S_2$,   
since the set of  second components of the elements in any generating set of $S^Y \rtimes T$ is a generating set of $T$. Combining these two facts gives precisely the assertion of the lemma.
\end{proof}

\begin{lem}\label{rank-wreath-monoid-group}
If $(X,S)$ is a transformation monoid and $(Y,G)$ is  a permutation group then 
\begin{equation}\label{eq-rank-wreath-group}
\rank(S^Y \rtimes G) \geq |Y|(\rank(S)-\rank(U_S)) + \rank(U_S^Y \rtimes G).
\end{equation}
\end{lem}
\begin{proof}
It is easy to check that  
 $$ S^Y \rtimes G=( (S^Y \setminus U_S^Y) \rtimes G )\cup (U_S^Y \rtimes G),$$
is a partition into two subsemigroups of $ S^Y \rtimes G$. Because $(S^Y \setminus U_S^Y) \rtimes G $ is an ideal, every generating set of $S^Y \rtimes G$ must contain a generating set of $U_S^Y \rtimes G$. To complete the proof, it is enough to show that every generating set of $S^Y \rtimes G$ has at least $|Y|(\rank(S)-\rank(U_S))$ elements in $(S^Y \setminus U_S^Y) \rtimes G$. Let $A$ be a generating set of $S^Y \rtimes G$. One can easily check  that, denoting by $\pi_1$ the projection on the first component,
$$ A'=\{ ^t\!f : f \in A\pi_1 , t\in G \}$$ is a generating set of $S^Y$. The equality  
$$ \rank(S^Y)=\rank(U_S^Y)+|Y|(\rank(S)-\rank(U_S))$$  has been proved in \cite[Theorem 1]{Wiegold:1987}. Hence, $A'$ has at least 
$$|Y|(\rank(S)-\rank(U_S))$$ elements in $S^Y\setminus U_S^Y$. On the other hand, if $f$ belongs to $U_S^Y$ and $t$ belongs to $G$ then $ ^t\!f \in U_S^Y$. Therefore, $A\pi_1$ must contain at least $$|Y|(\rank(S)-\rank(U_S))$$ elements in $S^Y\setminus U_S^Y$. This implies   that  $A$ has at least $$|Y|(\rank(S)-\rank(U_S))$$ elements in $(S^Y\setminus U_S^Y )\rtimes G$ and the proof is complete.   
\end{proof}

\begin{prop} \label{prop-rank-lower-bound}
Let $(X,S)$ and $(Y,T)$ be two transformation monoids. Then, the rank of 
$S^Y \rtimes T$ is greater than or equal to
\begin{equation}\label{eq-rank-wreath-general-1}
\rank(U_S^Y \rtimes U_T)+|Y|(\rank(S)-\rank(U_S))+\rank(T)-\rank(U_T).
\end{equation}
\end{prop}
\begin{proof}
This is straightforward using Lemmas  \ref{rank-wreath-monoids} and \ref{rank-wreath-monoid-group} .
\end{proof}

\begin{prop}\label{generating-set-minimum-size}
Let $(X,S)$ and $(Y,T)$ be two transformation monoids. Let $A'$, $A$ and $B$ be generating sets of minimum size of $U_S^Y \rtimes U_T$, $S$, and $T$, respectively. The set   
$$ C= A' \cup \{((a)_y,1): a \in A\setminus U_S , y \in Y \} \cup \{(\bar{1},b) : b\in B\setminus U_T\}$$  is a generating set of $S^Y \rtimes T$ with minimum size. Consequently, the rank of 
$S^Y \rtimes T$ is equal to
\begin{equation}\label{eq-rank-wreath-general-2}
\rank(U_S^Y \rtimes U_T)+|Y|(\rank(S)-\rank(U_S))+\rank(T)-\rank(U_T).
\end{equation}
\end{prop}
\begin{proof}
First, we show that 
$C$ is a generating set. Consider a pair $$(f,t) \in S^Y \rtimes T.$$ Because $B$ is a generating set of $T$, there exist $b_1,b_2,\ldots, b_k \in B$ such that $t=b_1b_2\ldots b_k$. This leads to the following factorization:  
\begin{equation}\label{factorization-2}
(f,t)=(f,1)(\bar{1},t)= \prod_{y \in Y} ((yf)_y,1) \prod_{i=1}^k (\bar{1},b_i).   
 \end{equation}
Because $A$ is a generating set of $S$ and $yf \in S$, for every $y \in Y$ there exist $a_{y1},a_{y2},\ldots, a_{yk_y} \in A$ such that $$yf= \prod_{i=1}^{k_y}a_{yi}.$$ Accordingly, we obtain the factorization 
\begin{equation}\label{factorization-3}
((yf)_y,1)=\prod_{i=1}^{k_y} ((a_{yi})_y,1).  
\end{equation}
Consider the pair $((a_{yi})_y,1)$ in \eqref{factorization-3}. If $a_{yi} \in U_S$ then $((a_{yi})_y,1) \in U_S^Y \rtimes U_T$ can be factorized into elements of $A'$; otherwise, $((a_{yi})_y,1) \in C$. This shows that the first product in  \eqref{factorization-2} can be rewritten in terms of elements of $C$. Now consider the pair $(\bar{1},b_i)$  in the second product in \eqref{factorization-2}. If $b_i \in U_T$ then $ (\bar{1},b_i) \in U_S^Y \rtimes U_T$ can be factorized into elements of  $A'$; otherwise, $(\bar{1},b_i) \in C$. This shows that the second product in \eqref{factorization-2} can be rewritten in terms of elements of $C$. Thus, $(f,t)$ can be factorized into elements of  $C$, whence $C$ is a generating set of $S^Y \rtimes T$,
which is the desired conclusion. Now, according to Proposition \ref{prop-rank-lower-bound}, the size of $C$ is equal to $\rank(S^Y \rtimes T)$.
\end{proof}

\begin{notation}  
 For a finite group $G$ denote by $\diam_{min}(G)$ \index[notation]{$\diam_{min}(G)$} the minimum of $\diam(G,A)$ over all generating sets of minimum size. 
\end{notation}

\begin{theorem}\label{th:upper-bound-general}
Given two transformation monoids $(X,S)$ and $(Y,T)$, there exist integers $0 \leq m_1 < N(S)$ and $0 \leq m_2 < N(T)$ such that 
\begin{align} \label{upper-bound-general}
 N(S^Y \rtimes T)
 \leq &(m_1+m_2) \diam_{min}(U_S^Y \rtimes U_T) \nonumber \\
 +& |Y|(N(S)-m_1)+N(T)-m_2.
\end{align}
\end{theorem}
\begin{proof}
Let $A$ and $B$ be generating sets of minimum size of $S$ and $T$, respectively,
such that $N(S,A)=N(S)$ and $N(T,B)=N(T)$. There exist $a_1,a_2,\ldots,a_{N(S)}
\in A$ and $b_1,b_2,\ldots,b_{N(T)}\in B$ such that $a_1a_2\ldots
a_{N(S)} \in \ker(S)$ and $b_1b_2\ldots b_{N(T)} \in \ker(T).$ Denote by $m_1$ and $m_2$ the number of invertible factors in the words $a_1a_2\ldots
a_{N(S)}$ and $b_1b_2\ldots b_{N(T)}$, respectively. 
Define the function $f$ from $Y$ to $\ker(S)$ to be the constant map with
image $a_1a_2\ldots a_{N(S)}$. By Lemma 
\ref{minimum ideal}, the pair $(f,b_1b_2\ldots b_{N(T)})$ is an element of the minimum ideal  of $S^Y \rtimes T$.
 
Let $A'$ be a generating set of $U_S^Y \rtimes U_T$ of minimum size such that $\diam(U_S^Y \rtimes U_T,A')=\diam_{min}(U_S^Y \rtimes U_T)$. By Proposition \ref{generating-set-minimum-size}, the set 
$$ C= A' \cup \{((a)_y,1): a \in A\setminus U_S , y \in Y \} \cup \{(\bar{1},b) : b\in B\setminus U_T\}$$  is a generating set of $S^Y \rtimes T$ of minimum size. To establish the inequality \eqref{upper-bound-general}, it is enough to show that the pair $(f,b_1b_2\ldots b_{N(T)})$ is a product of at most 
$$(m_1+m_2) \diam_{min}(U_S^Y \rtimes U_T) + |Y|(N(S)-m_1)+N(T)-m_2$$
 elements of $C$.
We have 
\begin{equation}\label{first-theorem-factorization1}
(f,b_1 b_2 \ldots b_{N(T)})=(f,1)(\bar{1},b_1b_2\ldots b_{N(T)})
= \prod_{i=1}^{N(S)}(\bar{a_i},1)\prod_{i=1}^{N(T)} (\bar{1},b_i).
\end{equation}
Consider the pair $(\bar{a_i},1)$ in the first product of \eqref{first-theorem-factorization1}.
If $ a_i \in A \setminus U_S $, then $$(\bar{a_i},1)=\prod_{y \in Y}({(a_i)}_y,1),$$ which is a product of $|Y|$ elements in $$\{((a)_y,1): a \in A\setminus U_S , y \in Y \}.$$ If  $a_i \in U_S$, then $ (\bar{a_i},1)$ can be written as a product of at most $\diam_{min}(U_S^Y\rtimes U_T)$ elements in $A'$. Accordingly, the first product in \eqref{first-theorem-factorization1} can be rewritten as a product of at most $$|Y|(N(S)-m_1)+m_1 \diam_{min}(U_S^Y \rtimes U_T)$$ elements in $C$.
Now consider the factor $(\bar{1},b_i)$ of the second product in \eqref{first-theorem-factorization1}. If $b_i \in B \setminus U_T$ then  $(\bar{1},b_i) \in C$; otherwise,  $(\bar{1},b_i) \in U_S^Y \rtimes U_T$ can be written as a product of at most $\diam_{min}(U_S^Y \rtimes U_T)$ elements in $A'$. Thus, the second product in \eqref{first-theorem-factorization1} can be rewritten as a product of at most $$N(T)-m_2 + m_2 \diam_{min}(U_S^Y \rtimes U_T)$$ elements in $C$. Combining these two facts shows that $(f,b_1b_2 \ldots b_{N(T)})$ can be written as a product of at most 
$$(m_1+m_2) \diam_{min}(U_S^Y \rtimes U_T) + |Y|(N(S)-m_1)+N(T)-m_2$$
 elements in $C$, which proves the theorem. 
\end{proof}
In the rest of this section we study some special cases.
\begin{theorem}\label{th:upper-bound-special}
Given two transformation monoids $(X,S)$ and $(Y,T)$, suppose that $T\neq \{1\}$ has trivial group of units and $|Y|=n$. Then the following inequality holds:
\begin{equation} \label{uper-bound-N1}
 N(S^Y \rtimes T) \leq \max \{n,\diam(U_S^Y,A' ) \} N(S)+ N(T),
\end{equation}
where $A'$ is a generating set of $U_S^Y$ with minimum size.
Furthermore, if $\rank(U_S^k)=k \, \rank(U_S) \mbox{ for } k \geq 1$,  then
\begin{equation} \label{upper-bound-N2}
 N(S^Y \rtimes T) \leq n N(S)+ N(T).
\end{equation}
 \end{theorem}

\begin{proof}
Let $A$ and $B$ be two generating sets of minimum size of $S$ and $T$, respectively,
such that $N(S,A)=N(S)$ and $N(T,B)=N(T)$. There exist $a_1,a_2,\ldots,a_{N(S)}
\in A$ and $b_1,b_2,\ldots,b_{N(T)}\in B\setminus \{1\}$ such that $$a_1a_2\ldots
a_{N(S)} \in \ker(S)$$ and $$b_1b_2\ldots b_{N(T)} \in \ker(T).$$ Define
the function $f$ from $Y$ to $\ker(S)$ to be the constant map with
image $a_1a_2\ldots a_{N(S)}$. By Lemma
\ref{minimum ideal}, the pair $(f,b_1b_2\ldots b_{N(T)})$ is an element of the minimum ideal  of $S^Y \rtimes T$. 
Let $A'$ be a generating set of $U_S^Y$ with minimum size. By Proposition \ref{generating-set-minimum-size}, the set 
$$ C'=(A' \times \{1\})  \cup \{((a)_y,1): a \in A\setminus U_S , y \in Y \} \cup (\{\bar{1}\} \times B\setminus \{1\})$$  is a generating set of $S^Y \rtimes T$ with minimum size. To establish the inequality \eqref{uper-bound-N1}, it is enough to show that the pair $(f,b_1b_2\ldots b_{N(T)})$ is a product of at most 
$$\max \{n,\diam(U_S^Y ,A') \} N(S)+ N(T)$$ elements of $C'$.
We have 
\begin{equation}\label{second-theorem-factorization1}
(f,b_1 b_2 \ldots b_{N(T)})=(f,1)(\bar{1},b_1b_2\ldots b_{N(T)})
= \prod_{i=1}^{N(S)}(\bar{a_i},1)\prod_{i=1}^{N(T)} (\bar{1},b_i).
\end{equation}
For $i=1,2,\ldots,N(T)$, the pair $(\bar{1},b_i)$ belongs to $C'$. Consider next the pairs $(\bar{a_j},1)$ with $$j=1,2,\ldots,N(S).$$
If $ a_j \in A \setminus U_S $, then $(\bar{a_j},1)=\prod_{y \in Y}({(a_j)}_y,1)$, which is a product of $n$ elements in $$\{((a)_y,1): a \in A\setminus U_S , y \in Y \}.$$ If  $a_j \in U_S$, then $ (\bar{a_j},1)$ can be written as a product of at most $\diam(U_S^Y,A')$ elements in $\{(g,1) : g\in A'\}$. 
Therefore, the product on the
rightmost side of \eqref{second-theorem-factorization1} can be rewritten as a product of at most 
$$\max\{n,\diam(U_S^Y,A')\} N(S) + N(T)$$ elements in $C'$ as we required. 

Consider the case where $\rank(U_S^Y)= |Y| \, \rank(U_S)$. By Proposition \ref{generating-set-minimum-size}, the set 
$$ C''= \{((a)_y,1): a \in A , y \in Y \} \cup \{(\bar{1},b) : b\in B\setminus \{1\}\}$$ 
is a generating set of $S^Y \rtimes T$ of minimum size. More precisely, since $U_T$ is trivial and $\rank(U_S^Y)= |Y| \, \rank(U_S)$, substituting $\rank(U_S^Y \rtimes U_T)$ by $ |Y| \, \rank(U_S)$ in formula \eqref{eq-rank-wreath-general-2} in Proposition \ref{generating-set-minimum-size}, gives $ |Y| \rank(S) + \rank(T)$ which is equal to $|C''|$.  We can factorize the pair $(f,b_1b_2 \ldots b_{N(T)})$ in $n N(S) + N(T) $ elements of $C''$ as follows:
\begin{equation} \label{factorization2}
(f,b_1 b_2 \ldots b_{N(T)})=(f,1)(\bar{1},b_1b_2\ldots b_{N(T)})
= \prod_{y \in Y}\prod_{i=1}^{N(S)}((a_i)_y,1)\prod_{i=1}^{N(T)} (\bar{1},b_i).
\end{equation}
This establishes the inequality \eqref{upper-bound-N2} and completes the proof of the theorem.
\end{proof}
 \section{Final remarks}

We collect here several of questions which remain open:

\begin{question}
In Lemma \ref{completely-regular} we have just found an upper bound for $M'(S)$ where $S$ is a completely regular semigroup. When does
 equality hold? What may we say for the other depth parameters? 
\end{question}

\begin{question}
 Theorem \ref{lower-bound-N'} gives a lower bound for $N'(S)$ where $S$ is a finite transformation semigroup. Similarly, it would be nice to find an upper bound for $M´(S)$ where $S$ is a finite transformation semigroup. 
\end{question}

\begin{question}
In Corollary \ref{N-transformation-cor1} the parameters $N$ and $N'$ are computed for the transformation semigroups $T_n, PT_n$ and $I_n$. What can we  say about $M,M'$ for them?
\end{question}

\begin{question}
The equalities $N = N'$ and $M = M'$ hold in all the semigroups which we have verified. Is there any example of a semigroup for which $N' < N$ and $M < M'$? 
\end{question}

\begin{question}
In Section \ref{se:depth}, we estimate the depth parameters for the families of transformation semigroups whose rank has been determined already in the literature. Other natural candidates that may be easy to verify are the semigroups $SP_n,SPO_n$ or semigroups of orientation preserving transformations such as $POP_n, OP_n$ or $ POPI_n$.
\end{question}

\begin{question}
We have established upper bounds for $N(S)$ where $S$ is a direct product or wreath product of two finite monoids. It would be interesting to obtain analogous results for the
other depth parameters.
\end{question}

\begin{question}
Give examples to show that the inequalities in Theorems \ref{th:upper-bound-direct-product}, \ref{th:upper-bound-general} and \ref{th:upper-bound-special} may not be improved. 
\end{question}

\section{Acknowledgments}
This is part of the author's Ph.D. thesis, written under the supervision of Professors Jorge Almeida and Pedro Silva at the University of Porto with the financial support from FCT (Funda\c{c}\~ao para a Ci\^encia e a Tecnologia) with reference  SFRH /BD/51170 /2010. The author wishes to express her thanks to her supervisors for suggesting the problem and for many stimulating conversations.  

\bibliographystyle{amsplain}
\bibliography{nasimref}

\providecommand{\bysame}{\leavevmode\hbox to3em{\hrulefill}\thinspace}
\providecommand{\MR}{\relax\ifhmode\unskip\space\fi MR }
\providecommand{\MRhref}[2]{%
  \href{http://www.ams.org/mathscinet-getitem?mr=#1}{#2}
}
\providecommand{\href}[2]{#2}
\begin{thebibliography}{10}

\bibitem{Almeida:1994}
J.~Almeida, \emph{Finite semigroups and universal algebra}, Series in Algebra,
  vol.~3, World Scientific Publishing Co., Inc., River Edge, NJ, 1994,
  Translated from the 1992 Portuguese original and revised by the author.
  \MR{1331143 (96b:20069)}

\bibitem{Cerny:1964}
J.~{\v{C}}ern{\'y}, \emph{A remark on homogeneous experiments with finite
  automata}, Mat.-Fyz. \v Casopis Sloven. Akad. Vied \textbf{14} (1964),
  208--216. \MR{0168429 (29 \#5692)}

\bibitem{Fernandes:2001}
V.~H. Fernandes, \emph{The monoid of all injective order preserving partial
  transformations on a finite chain}, Semigroup Forum \textbf{62} (2001),
  no.~2, 178--204. \MR{1831507 (2002a:20075)}

\bibitem{Garba:1990}
G.~U. Garba, \emph{Idempotents in partial transformation semigroups}, Proc.
  Roy. Soc. Edinburgh Sect. A \textbf{116} (1990), no.~3-4, 359--366.
  \MR{1084739 (92c:20123)}

\bibitem{Gomes&Howie:1987}
M.~S. Gomes and J.~M. Howie, \emph{On the ranks of certain finite semigroups of
  transformations}, Math. Proc. Cambridge Philos. Soc. \textbf{101} (1987),
  no.~3, 395--403. \MR{878889 (88e:20057)}

\bibitem{Gomes&Howie:1992}
\bysame, \emph{On the ranks of certain semigroups of order-preserving
  transformations}, Semigroup Forum \textbf{45} (1992), no.~3, 272--282.
  \MR{1179851 (93h:20070)}

\bibitem{Higgins:1992}
P.~M. Higgins, \emph{Techniques of semigroup theory}, Oxford University Press,
  New York, 1992.

\bibitem{Howie&McFadden:1990}
J.~Howie and B.~McFadden, \emph{Idempotent rank in finite full transformation
  semigroups}, Royal Soc. (Edinburgh), no. 114, 1990, pp.~161--167.

\bibitem{Howie:1966}
J.~M. Howie, \emph{The subsemigroup generated by the idempotents of a full
  transformation semigroup}, J. London Math. Soc. 41 (1966), 707--716.

\bibitem{Hungerford:1980}
T.~W. Hungerford, \emph{Algebra}, Graduate Texts in Mathematics, vol.~73,
  Springer-Verlag, New York-Berlin, 1980, Reprint of the 1974 original.
  \MR{600654 (82a:00006)}

\bibitem{Kari:2001}
J.~Kari, \emph{A counter example to a conjecture concerning synchronizing words
  in finite automata}, Bull. Eur. Assoc. Theor. Comput. Sci. EATCS (2001),
  no.~73, 146. \MR{1835978}

\bibitem{Pin:1978}
J.~E. Pin, \emph{Le probl\`eme de la synchronisation et la conjecture de \v
  {C}ern\'y}, Noncommutative structures in algebra and geometric combinatorics
  ({N}aples, 1978), Quad. ``Ricerca Sci.'', vol. 109, CNR, pp.~37--48.

\bibitem{Pin:1986}
\bysame, \emph{Varieties of formal languages}, North Oxford, London and Plenum,
  New York, 1986.

\bibitem{Rhodes&Steinberg:2009}
J.~Rhodes and B.~Steinberg, \emph{The {$q$}-theory of finite semigroups},
  Springer Monographs in Mathematics, Springer, New York, 2009. \MR{2472427
  (2010h:20132)}

\bibitem{Ruskuc&Robertson&Thomson:2003}
E.~F. Robertson, N.~Ru{\v{s}}kuc, and M.~R. Thomson, \emph{Finite generation
  and presentability of wreath products of monoids}, J. Algebra \textbf{266}
  (2003), no.~2, 382--392. \MR{1995120 (2004f:20098)}

\bibitem{Rystsov:1992}
I.~C. Rystsov, \emph{On the rank of a finite automaton}, Kibernet. Sistem.
  Anal. (1992), no.~3, 3--10, 187. \MR{1190855 (93h:68097)}

\bibitem{Wiegold:1987}
J.~Wiegold, \emph{Growth sequences of finite semigroups}, J. Austral. Math.
  Soc. (Ser.A) 43 (1987), 16--20.

\end{thebibliography}

\end{document}